\documentclass[a4paper,12pt]{article}
\usepackage{amsmath}
\usepackage{amssymb}
\usepackage{amsthm}
\usepackage{amsfonts}
\usepackage{enumerate}
\usepackage[pdfstartview=FitH]{hyperref}
\allowdisplaybreaks

\newtheorem{theorem}{Theorem}[section]
\newtheorem{lem}{Lemma}[section]
\newtheorem{prp}[theorem]{Proposition}
\newtheorem{thm}[theorem]{Theorem}

\newtheorem{dfn}{Definition}[section]

\newtheorem{remark}{Remark}

\begin{document}

\title{Random dynamics of two-dimensional stochastic second grade fluids}

\author{ Shijie Shang$^{1}$ }

\footnotetext[1]{\ School of Mathematical Sciences, University of Science and Technology of China, 230026 Hefei, China. Email: ssjln@mail.ustc.edu.cn}
\date{}

\maketitle

\begin{abstract}
In this paper, we consider a stochastic model of incompressible non-Newtonian fluids of second grade on a bounded domain of $\mathbb{R}^2$ with multiplicative noise. We first show that the solutions to the stochastic equations of second grade fluids generate a continuous random dynamical system. Second, we investigate the Fr\'{e}chet differentiability of the random dynamical system.
Finally, we establish the asymptotic compactness of the random dynamical system, and the existence of random attractors for the random dynamical system, we also obtain the upper semi-continuity of the perturbed random attractors when the noise intensity approaches zero.

\end{abstract}



\vspace{3mm}

\numberwithin{equation}{section}
\section{Introduction}
Recently, non-Newtonian fluids have attracted more and more attention. In this article, we consider a class of non-Newtonian fluids of differential type, namely the second grade fluids, which is an admissible model of slow flow fluids such as industrial fluids, slurries, polymer melts, etc.
The stochastic model is given as follows:
\begin{align}\label{1.a}
\left\{
\begin{aligned}
& d(u-\alpha \Delta u)+ \Big(-\nu \Delta u+{\rm curl}(u-\alpha \Delta u)\times u+\nabla\mathfrak{P}\Big)\,dt \\
&\quad = F(u)\,dt+(u-\alpha \Delta u)\circ \epsilon dW,\quad \rm{ in }\ \mathcal{O}\times(0,+\infty), \\
&\begin{aligned}
& ({\rm{div}}\,u)(t,x)=0, \quad &&\rm{in}\ \mathcal{O}\times(0,+\infty), \\
& u(t,x)=0,  &&\rm{in}\ \partial \mathcal{O}\times(0,+\infty), \\
& u(0,x)=f(x),  &&\rm{in}\ \mathcal{O},
&\end{aligned}
\end{aligned}
\right.
\end{align}
where $\mathcal{O}$ is a bounded domain of $\mathbb{R}^2$ with boundary $\partial \mathcal{O}$ of class $\mathcal{C}^{3,1}$. $u=(u_1,u_2)$ and $\mathfrak{P}$ represent the random velocity and the modified pressure, respectively. $\alpha$ is a positive constant and $\nu$ is the kinematic viscosity. $\epsilon$ is a nonzero real-valued parameter.
$W$ is a one-dimensional two-sided Brownian motion which will be specified later.
%
%
The fluid is driven by the external force $F(u)\,dt$ and the noise $(u-\alpha \Delta u)\circ\epsilon dW$, here $\circ$ denotes that the stochastic integral is interpreted in the Stratonovich sense. The initial value $f$ is a deterministic function on $\mathcal{O}$.


The model of second grade fluids reduces to Navier-Stokes equations(NSEs) when $\alpha=0$, and it's also a good approximation of the NSEs as shown in \cite{2016-Arada-p2557-2586,2002-Iftimie-p83-86}. Furthermore, it has interesting connections with other fluid models, see \cite{2003-Busuioc-p1119-1149,2001-Shkoller-p539-543} and references therein. We refer readers to \cite{1997-Cioranescu-p317-335,1995-Dunn-p689-729,1979-Fosdick-p145-152} for a comprehensive theory of the deterministic case.
For the stochastic case, we refer to \cite{2012-Razafimandimby-p4251-4270,2017-Shang-p-,2016-Wang-p196-213,2017-Zhai-p471-498} and references therein.

The purpose of this paper is to investigate dynamics of the stochastic model for the second grade fluids. In particular, we show that solutions of (\ref{1.a}) generate a continuous random dynamical system/continuous perfect cocycle, and we prove that the cocycle is Fr\'{e}chet differentiable at the points with higher regularity of the phase space. Then we establish that the random dynamical system is asymptotically compact and possesses a random attractor, and we also prove the upper semi-continuity of the perturbed random attractors as the noise intensity $\epsilon\rightarrow 0$.


To get our main results, we construct a version $u(t,f,\omega)$ of solutions to (\ref{1.a}) as follows:
\begin{align}\label{add 0325.1}
  u(t,f,\omega)=Q(t,\omega)v(t,f,\omega),
\end{align}
where $Q(t,\omega):=\mathrm{e}^{\epsilon W(t,\omega)}$, $v(t,f,\omega)$ satisfies a partial differential equation(PDE) with random coefficients.
We can show that this version $u$ generates a random dynamical system. And it is relatively convenient to use the transform (\ref{add 0325.1}) to investigate the continuity and the Fr\'{e}chet differentiability of the random dynamical system.
However, note that the transform (\ref{add 0325.1}) is non-stationary(see \cite{2003-Duan-p2109-2135,2004-Flandoli-p1385-1420} for stationary conjugations, which transform a stochastic partial differential equation(SPDE) into a random PDE and the two corresponding random dynamical systems are equivalent), and $v(t,f,\omega)$ even does not generate a random dynamical system.


The motivation of investigating Fr\'{e}chet derivatives of the solution $u(t,f,\omega)$ is mainly based on the following two considerations. Firstly,
in the theory of random invariant manifolds, to study dynamics near an equilibrium(or stationary random point) of a cocycle, the usual method is to analyse the spectrum for the corresponding linearized cocycle, which demands the cocycle to be sufficiently Fr\'{e}chet differentiable, see e.g. \cite{2010-Mohammed-p3543-3591,2008-Mohammed-p105-105}.
Secondly, from the anticipating/nonadapted stochastic analysis point of view,
if the initial value of (\ref{1.a}) is a random variable $\xi$(not deterministic) and measurable with respect to the $\sigma$-algebra $\sigma(W(t): 0\leq t\leq T)$ for some $T>0$, then the substitute process $u(t,\xi(\omega),\omega)$ will still be a solution of the anticipating SPDE (\ref{1.a}) with the anticipating initial value $u(0)=\xi$, the proof of such a substitution result depends on the Fr\'{e}chet differentiability of the solution $u(t,f,\omega)$ to (\ref{1.a}), see \cite{2018-Shang-p1138-1152}.
Since the curl-term in (\ref{1.a}) is order of three, the system (\ref{1.a}) is highly nonlinear(far from semi-linear). This leads to that the dissipation of (\ref{1.a}) is much weaker than that of 2D stochastic NSEs and other SPDEs considered in \cite{2010-Mohammed-p3543-3591} and \cite{2008-Mohammed-p105-105}, which have the smoothing effect in finite time.
To overcome this difficulty, we adopt the energy-equation approach (see \cite{1998-Moise-p1369-1393} or Lemma \ref{add 0323 lemma.1}) to prove the continuity of the solution $u(t,f,\omega)$.
Due to the weak dissipation,
we only establish that the solution $u(t,f,\omega)$ is Fr\'{e}chet differentiable at the points with higher regularity of the phase space, but don't know how to figure out the Fr\'{e}chet differentiability on the whole phase space.
As far as we are aware, our results about the Fr\'{e}chet differentiability are new even in the deterministic case.



Random attractors capture the asymptotic behavior of a random dynamical system.
The existence of random attractors has been studied in a lot of literature, see, e.g.
\cite{1998-Arnold-p586-586,1997-Crauel-p307-341,1994-Crauel-p365-393,1996-Flandoli-p21-45}
and the references therein.
In \cite{1994-Crauel-p365-393}, the transform (\ref{add 0325.1}) is also used to establish the existence of random attractors for NSEs on a bounded domain of $\mathbb{R}^2$ with linear multiplicative noise.
However, as we have already pointed out, the solution operator of (\ref{1.a}) is not smoothing or compact, so it is difficult to directly construct a compact invariant random set absorbing every bounded tempered random set.
In the deterministic case, Moise, Rosa and Wang \cite{1998-Moise-p1369-1393} overcame this difficulty by the energy-equation approach, where the authors showed that the corresponding dynamical system is asymptotically compact and possesses a compact global attractor under the constant external force.
Inspired by this, we establish the existence of random attractors for the random dynamical system by the following two steps(see \cite{2006-Bates-p1-21,2006-Caraballo-p484-498} or Lemma \ref{randomattractor}).
First, by establishing a supplementary property of the cocycle $u(t,f,\omega)$(see Lemma \ref{add 0328 lemma.1}),
we show that the random dynamical system admits a closed bounded tempered random absorbing set, under a smallness assumption on the Lipschitz constant of the external force $F$.
Second, we show that the random dynamical system is asymptotically compact in the stochastic setting. To achieve the asymptotic compactness, we construct a subsequence by the diagonal method, and then use the energy-equation for $u(t,f,\omega)$ established through the transform (\ref{add 0325.1}), to obtain the strong convergence of this subsequence.
Lastly, we establish the upper semi-continuity of the perturbed random attractors of (\ref{1.a}), i.e.
the perturbed random attractors converge to the corresponding deterministic global attractor in the sense of Hausdorff semi-metric as the random perturbations approach zero.
To show the precompactness of the union of perturbed random attractors, which is an important condition to establish the upper semi-continuity(see \cite{2009-Wang-p1-18} or Lemma \ref{add 0409.8}), we appeal again to the asymptotic compactness of the random dynamical systems.

The organization of this paper is as follows. Section 2 is to introduce some preliminaries.
In Section 3, we formulate the hypotheses and state our main results. In Section 4, we show that solutions of (\ref{1.a}) generate a continuous random dynamical system. Section 5 is devoted to the Fr\'{e}chet differentiability of the solution. In Section 6, we establish the asymptotic compactness of the random dynamical system, the existence and the upper semi-continuity of random attractors.


Throughout this paper, $C,C(T,\omega),...$ are positive constants whose value may be different from line to line and which may depend on some parameters $T,\omega,...$.


%


\section{Preliminaries}

In this section, we will introduce some functional spaces and preliminaries needed in the paper.

For $p\geq 1$ and $k\in\mathbb{N}$, we denote by $L^p(\mathcal{O})$
and $W^{k,p}(\mathcal{O})$ the usual $L^p$ and Sobolev spaces over $\mathcal{O}$ respectively, and write $H^k(\mathcal{O}):=W^{k,2}(\mathcal{O})$.
We write $\mathbb{X}=X\times X$ for any vector space $X$.
The set of all divergence free and infinitely differentiable functions with compact support in $\mathcal{O}$ is denoted by $\mathcal{C}$.
$\mathbb{H}$ (resp. $\mathbb{V}$) is the completion of $\mathcal{C}$ in $\mathbb{L}^2(\mathcal{O})$ (resp. $\mathbb{H}^1(\mathcal{O})$).
We denote by $(\cdot,\cdot)$ and $|\cdot|$ the inner product in $\mathbb{L}^2(\mathcal{O})$(in $\mathbb{H}$) and the induced norm, respectively.
Let $((u,v)):=\int_\mathcal{O}\nabla u\cdot\nabla vdx$, where $\nabla$ is the gradient operator.
%
%
%
%
Set $\|u\|:=((u,u))^{\frac{1}{2}}$.
We endow the space $\mathbb{V}$ with the norm $|\cdot|_{\mathbb{V}}$ generated by the inner product $(u,v)_\mathbb{V}:=(u,v)+\alpha ((u,v))$, for any $u,v\in\mathbb{V}$.
%
Let $\mathcal{P}>0$ be the $\rm Poincar\acute{e}$ constant for the domain $\mathcal{O}$, i.e. $\mathcal{P}$ is the optimal constant in the following $\rm Poincar\acute{e}$ inequality
\begin{align}\label{add 0329.1}
  |v|\leq \mathcal{P} |v|_\mathbb{V},\quad \forall\, v\in\mathbb{V}.
\end{align}
(\ref{add 0329.1}) implies that
\begin{align}\label{Poincare}
(\mathcal{P}^2+\alpha)^{-1}|v|^2_\mathbb{V} \leq \|v\|^2
\leq\alpha^{-1}|v|^2_\mathbb{V},\quad \forall\, v\in\mathbb{V}.
\end{align}
We also introduce the space $\mathbb{W}:=\big\{u\in\mathbb{V}: {\rm curl}(u-\alpha\Delta u)\in L^2(\mathcal{O})\big\}$,
and endow it with the semi-norm $|\cdot|_{\mathbb{W}}$ generated by the scalar product $(u,v)_\mathbb{W}:=\big({\rm curl}(u-\alpha\Delta u),{\rm curl}(v-\alpha\Delta v)\big)$. In fact, $\mathbb{W}=\mathbb{H}^3(\mathcal{O})\cap\mathbb{V}$, and this semi-norm $|\cdot|_{\mathbb{W}}$ is equivalent to the usual norm in $\mathbb{H}^3(\mathcal{O})$, the proof can be found in \cite{1997-Cioranescu-p317-335}.


Identifying the Hilbert space $\mathbb{V}$ with its dual space $\mathbb{V}^*$, via the Riesz representation, we consider the system (\ref{1.a}) in the framework of Gelfand triple:
$\mathbb{W}\subset \mathbb{V}\subset\mathbb{W}^*$.
%
\noindent We also denote by $\langle \cdot,\cdot\rangle$ the dual relation between $\mathbb{W}^*$ and $\mathbb{W}$ from now on.

Because the injection of $\mathbb{W}$ into $\mathbb{V}$ is compact, there exists a sequence $\{e_i\}_{i=1}^{\infty}$ of elements of $\mathbb{W}$ which forms an orthonormal basis in $\mathbb{W}$ and an orthogonal basis in $\mathbb{V}$, such that
\begin{align}\label{Basis}
(u,e_i)_{\mathbb{W}}=\lambda_i(u,e_i)_{\mathbb{V}},\quad\forall\,u\in\mathbb{W},
\end{align}
where $0<\lambda_i\uparrow\infty$. Since $\partial \mathcal{O}$ is of class $\mathcal{C}^{3,1}$, Lemma 4.1 in \cite{1997-Cioranescu-p317-335} implies that $e_i\in \mathbb{H}^4(\mathcal{O}),\ \forall\,i\in\mathbb{N}$.

Define the Stokes operator by $Au:=-\Pi_{\mathbb{H}}\Delta u, \forall\,u\in D(A)=\mathbb{H}^2(\mathcal{O})\cap\mathbb{V}$, here the map $\Pi_{\mathbb{H}}:\mathbb{L}^2(\mathcal{O})\rightarrow\mathbb{H}$ is the usual Helmholtz-Leray projection, that is, for any $u\in \mathbb{L}^2(\mathcal{O})$, there exist $\phi_u\in\mathbb{H}^1(\mathcal{O})$ and a unique $\Pi_{\mathbb{H}}u\in\mathbb{H}$ such that $u=\Pi_{\mathbb{H}}u+\nabla\phi_u$. Hence
\begin{align}
\label{curlP=curl} {\rm curl}(\Pi_{\mathbb{H}}v)={\rm curl}(v),\quad \forall\ v\in\mathbb{H}^1(\mathcal{O}),\\
\label{W equivalence norm} |u|_\mathbb{W}=|{\rm curl}[(I+\alpha A)u]|,\quad \forall\ u\in\mathbb{W}.
\end{align}

\noindent Since the Stokes operator $A$ is self-adjoint and positive on $\mathbb{H}$, the spectrum $\sigma(A)\subset [0,\infty)$, and the resolvent $(I+\alpha A)^{-1}$ is a linear bounded operator on $\mathbb{H}$.
Furthermore, the Lax-Milgram theorem implies that $(I+\alpha A)^{-1}u\in\mathbb{V}$ for any $u\in\mathbb{H}$(or see Theorem \uppercase\expandafter{\romannumeral1}.2.1 in \cite{1977-Temam-p500-500}).
By a straightforward calculation, we have
\begin{align}\label{curl and V}
\big|{\rm curl}(v)\big|^2\leq \frac{2}{\alpha}|v|_{\mathbb{V}}^2,\quad \forall\, v\in\mathbb{V}.
\end{align}
Therefore, from (\ref{W equivalence norm}) and (\ref{curl and V}), we obtain that the operator $(I+\alpha A)^{-1}$ also defines an isomorphism from $\mathbb{H}^1(\mathcal{O})\cap\mathbb{H}$
into $\mathbb{W}$. Moreover,
\begin{gather}
\label{inverse operator transform inequality}|(I+\alpha A)^{-1}v|_\mathbb{W}\leq C |v|_{\mathbb{V}}, \quad \forall\, v\in\mathbb{V}, \\
\label{inverse operator transform} \big((I+\alpha A)^{-1}u,v\big)_\mathbb{V}=(u,v), \quad \forall\, u,v\in\mathbb{V}.
\end{gather}

\noindent Since the injection of $\mathbb{W}$ into $\mathbb{V}$ is compact, from (\ref{inverse operator transform}) we see that $(I+\alpha A)^{-1}$ is also a self-adjoint positive compact operator from $\mathbb{V}$ to $\mathbb{V}$. What's more, the operator norm $\|(I+\alpha A)^{-1}\|_{L(\mathbb{V},\mathbb{V})}$ is given by
\begin{align}\label{add 0329.2}
  \|(I+\alpha A)^{-1}\|_{L(\mathbb{V},\mathbb{V})}=\sup_{v\in\mathbb{V},v\neq 0}\frac{((I+\alpha A)^{-1}v,v)_\mathbb{V}}{(v,v)_\mathbb{V}} = \sup_{v\in\mathbb{V},v\neq 0}\frac{|v|^2}{|v|^2+\alpha\|v\|^2} = \frac{\mathcal{P}^2}{\mathcal{P}^2+\alpha} .
\end{align}


\noindent Set $\widehat{A}:=(I+\alpha A)^{-1}A$. Then $\widehat{A}$ is a
continuous linear operator from $\mathbb{W}$ to itself, and satisfies
\begin{gather}
\label{A transform 01}(\widehat{A}u,v)_\mathbb{V}=(Au,v)=((u,v)), \quad \forall\,u\in\mathbb{W},\ v\in\mathbb{V},\\
\label{A transform 03}|\widehat{A}u|_{\mathbb{W}^*}\leq C|\widehat{A}u|_{\mathbb{V}}\leq C|u|_{\mathbb{V}},\quad\forall\,u\in\mathbb{W}.
\end{gather}


Define the bilinear operator $\widehat{B}(\cdot\,,\cdot):\ \mathbb{W}\times\mathbb{V}\rightarrow\mathbb{W}^*$ by
\begin{align}\label{definition of B hat}
\widehat{B}(u,v):=(I+\alpha A)^{-1}\mathbb{P}\big({\rm curl}(u-\alpha \Delta u)\times v\big).
\end{align}
For simplicity, we write $\widehat{B}(u):=\widehat{B}(u,u)$. The following lemma can be found in \cite{2012-Razafimandimby-p4251-4270}.
\begin{lem}\label{Lem-B-01}
\begin{align*}
&|\widehat{B}(u,v)|_{\mathbb{W}^*}\leq C|u|_\mathbb{W}|v|_\mathbb{V},\quad\forall\,u\in\mathbb{W},\ v\in\mathbb{V},\\
&|\widehat{B}(u,u)|_{\mathbb{W}^*}\leq C|u|^2_\mathbb{V},\quad\forall\,u\in\mathbb{W},\\
&\langle\widehat{B}(u,v),v\rangle=0, \quad\forall\,u, v\in\mathbb{W},\\
&\langle\widehat{B}(u,v),w\rangle=-\langle\widehat{B}(u,w),v\rangle,\quad\forall\,u, v, w\in\mathbb{W}.
\end{align*}
\end{lem}

From the above lemma, the following estimate can be easily derived,
\begin{align}\label{add 0408.1}
  |\widehat{B}(u,v)|_{\mathbb{V}}\leq C|u|_\mathbb{W}|v|_\mathbb{W},\quad\forall\,u, v\in\mathbb{W}.
\end{align}
%
%
If $u\in\mathbb{W}\cap \mathbb{H}^4(\mathcal{O})$ and $v\in\mathbb{W}$, then $\widehat{B}(u,v)\in\mathbb{W}$. By (\ref{curlP=curl}) and (\ref{W equivalence norm}), we have for $w\in\mathbb{W}$,
\begin{align}\label{Eq B-05}
\begin{aligned}
\big(\widehat{B}(u,v),w\big)_{\mathbb{W}}
=&\Big({\rm curl}\big[\mathbb{P}\big({\rm curl}\big(u-\alpha\Delta u\big)\times v\big)\big],{\rm curl}\big(w-\alpha\Delta w\big)\Big) \\
=&\Big({\rm curl}\big[{\rm curl}\big(u-\alpha\Delta u\big)\times v\big],{\rm curl}\big(w-\alpha\Delta w\big)\Big)\\
=& \Big(\big(v\cdot\nabla\big){\rm curl}\big(u-\alpha\Delta u\big),{\rm curl}\big(w-\alpha\Delta w\big)\Big).
\end{aligned}
\end{align}

\noindent In particular, we have the following new identity, which is very useful to estimates of $\mathbb{W}$-norm.
\begin{align}\label{Eq B-06}
\begin{aligned}
\big(\widehat{B}(u,v),u\big)_{\mathbb{W}}
=&\Big(\big(v\cdot\nabla\big){\rm curl}\big(u-\alpha\Delta u\big),{\rm curl}\big(u-\alpha\Delta u\big)\Big) \\ =& 0, \quad \forall\,u\in\mathbb{W}\cap \mathbb{H}^4(\mathcal{O}),\, v\in\mathbb{W}.
\end{aligned}
\end{align}

Next, we introduce some preliminaries about random dynamical systems, see \cite{1998-Arnold-p586-586,1997-Crauel-p307-341,1994-Crauel-p365-393,1996-Flandoli-p21-45} for more details.
Throughout this paper, we take the complete probability space $(\Omega,\mathcal{F},P)$ as that
$\Omega:=\{\omega\in C(\mathbb{R};\mathbb{R})|\omega(0)=0\}$ is equipped with the compact open topology, $\mathcal{F}$ is the completion of the Borel $\sigma$-algebra of $\Omega$ with respect to the Wiener measure $P$.
We identify the one-dimensional two-sided Brownian motion $W$ with the canonical Wiener process, i.e.
\begin{align}
  W(t,\omega)=\omega(t), \quad t\in\mathbb{R}, \omega\in\Omega.
\end{align}
Let $\theta:\mathbb{R}\times\Omega\rightarrow\Omega$
be the Wiener shift on $(\Omega,\mathcal{F},P)$, i.e. $\theta(t,\omega)(s):=\omega(t+s)-\omega(t)$, $t,s\in\mathbb{R},\ \omega\in\Omega$.
Let $(X,\|\cdot\|_{X})$ be a separable Hilbert space with the Borel $\sigma$-algebra $\mathcal{B}(X)$.

%
%

\begin{dfn}
  A continuous random dynamical system or continuous perfect cocycle on $X$ is a $\mathcal{B}(\mathbb{R}^+)\otimes\mathcal{B}(X)\otimes\mathcal{F}/\mathcal{B}(X)$ measurable map $\varphi:\mathbb{R}^+\times X\times\Omega\rightarrow X,\quad (t,x,\omega)\mapsto \varphi(t,x,\omega)$, and for all $\omega\in\Omega$,
  \begin{itemize}
    \item [(i)]  $\varphi(0,x,\omega)=x ,\quad\forall\, x\in X$.
    \item [(ii)]  (cocycle property) $\varphi(t,\varphi(s,x,\omega),\theta(s,\omega))=\varphi(t+s,x,\omega), \quad\forall\,t,s\geq 0,x\in X$.
    \item [(iii)]  The map $\mathbb{R}^{+}\times X\ni (t,x)\mapsto \varphi(t,x,\omega)\in X$ is continuous.
  \end{itemize}
\end{dfn}


For nonempty sets $A,B\in 2^X$, we set
\begin{align}\label{add 0410.1}
  d(A,B):=\sup_{x\in A}\inf_{y\in B}\|x-y\|_X, \quad d(x,B):=d(\{x\},B),\quad d(B):=\sup_{x\in B}\|x\|_X,
\end{align}
here $d(A,B)$ is called the Hausdorff semi-metric between $A$ and $B$.

\begin{dfn} Let $\varphi$ be a random dynamical system on $X$.
  \item [(1)]  A set-valued map $B:\Omega\rightarrow 2^X\backslash\emptyset$ is said to be a random set if the map $\omega\mapsto d(x,B(\omega))$ is measurable for any $x\in X$.
      If $B$ is a random set, and $B(\omega)$ is closed (compact) for a.e. $\omega\in\Omega$, then $B$ is called a random closed (compact) set.
      A random set $B$ is said to be bounded if there exists a random variable $R(\omega)\geq 0$ such that for a.e. $\omega\in\Omega$,
      \begin{align*}
         B(\omega)\subset\{x\in X: \|x\|_X\leq R(\omega)\}.
      \end{align*}
  \item [(2)]  A random set $B$ is called tempered if for a.e. $\omega\in\Omega$,
      \[\lim_{t\rightarrow +\infty} e^{-\beta t}d(B(\theta(-t,\omega)))=0, \quad \forall\, \beta>0.\]
  \item [(3)]  Let $A,B$ be random sets. $A$ is said to absorb $B$ if for a.e. $\omega\in\Omega$, there exists an absorption time $t_B(w)$ such that for any $t\geq t_B(\omega)$, $\varphi(t,B(\theta(-t,\omega)),\theta(-t,\omega))\subset A(\omega)$. $A$ is said to attract $B$ if a.e. $\omega\in\Omega$,
      \begin{align}
        d\Big(\varphi\big(t,B(\theta(-t,\omega)),\theta(-t,\omega)\big),A(\omega)\Big)\rightarrow 0, \quad  \text{ as } t\rightarrow\infty .
      \end{align}
  \item [(4)]  A random set $G$ is called a random absorbing set, if it absorbs every bounded tempered random set of $X$.
  \item [(5)]  $\varphi$  is said to be asymptotically compact in $X$ if for any bounded tempered random set $D$,  for a.e. $\omega\in\Omega$, $\{\varphi(t_n,x_n,\theta(-t_n,\omega))\}_{n=1}^{\infty}$ has a convergent subsequence in $X$ whenever $t_n\rightarrow\infty$, and $x_n\in D(\theta(-t_n,\omega))$.
  \item [(6)]  A random compact set $A(\omega)$ is said to be the random attractor of $\varphi$ in $X$ if it attracts every bounded tempered random set of $X$ and $\varphi(t,A(\omega),\omega)=A(\theta(t,\omega))$ for a.e. $\omega\in\Omega$ and all $t\geq0$.
\end{dfn}


The following theorem is an important result to obtain the existence of random attractors, see Proposition 4.1 in \cite{2006-Bates-p1-21} or Theorem 7 in \cite{2006-Caraballo-p484-498}.

\begin{lem}\label{randomattractor}
 Let $\varphi$ be a continuous random dynamical system of $X$, if $\varphi$ is asymptotically compact, and there exists a closed bounded tempered random absorbing set, then $\varphi$ has a unique random attractor in $X$.
\end{lem}

Finally, we present an abstract result for establishing the upper semi-continuity of a family of random attractors, see Theorem 3.1 in \cite{2009-Wang-p1-18}.
\begin{lem}\label{add 0409.8}
  Let $\{\varphi^\epsilon\}_{0<|\epsilon|< 1}$ be a family of continuous random dynamical system of $X$ with the random attractors $\{\mathcal{A}^{\epsilon}\}_{0<|\epsilon|< 1}$ and random absorbing sets $\{G^{\epsilon}\}_{0<|\epsilon|< 1}$, $\varphi_0$ is a deterministic continuous dynamical system of $X$ with the attractor $\mathcal{A}^0$(i.e. compact and invariant and attract every bounded set of $X$). Assume that the following conditions are satisfied.
\begin{itemize}
  \item[(i)] For any $t\geq 0$ and a.e. $\omega\in\Omega$, $\varphi^{\epsilon_n}(t,x_n,\omega)\rightarrow\varphi^0(t,x)$ as $n\rightarrow\infty$, whenever $\epsilon_n\rightarrow 0$ and $x_n\rightarrow x$ in $X$.
  \item[(ii)] There exists a deterministic constant $C$ such that
    $\limsup_{\epsilon\rightarrow 0} d(G^{\epsilon}(\omega))\leq C$ for a.e. $\omega\in\Omega$ .
  \item[(iii)] There exists $\epsilon_0>0$ such that for a.e. $\omega\in\Omega$, the union
$\bigcup_{0<|\epsilon|<\epsilon_0}\mathcal{A}^{\epsilon}(\omega)$ is precompact in  X.
\end{itemize}
Then for a.e. $\omega\in\Omega$, $d(\mathcal{A}^{\epsilon}(\omega),\mathcal{A}^0)\rightarrow 0$ as $\epsilon\rightarrow 0$, here $d$ is the Hausdorff semi-metric. In this case, we say that the family of random attractors $\{\mathcal{A}^{\epsilon}\}_{0<|\epsilon|<1}$ is upper semi-continuous as $\epsilon\rightarrow 0$.

\end{lem}

\section{Hypotheses and main results}

In this section, we will formulate precise assumptions on the external force $F$,
and state the main results of this paper.

Let $F:\mathbb{V}\rightarrow\mathbb{V}$
be a given measurable map. We introduce the following Lipschitz condition of $F$:

\noindent {\bf(F1)} There is a constant $C_F$ such that
\begin{align}\label{add 0321.1}
|F(u_1)-F(u_2)|_\mathbb{V}\leq C_F|u_1-u_2|_\mathbb{V},\quad\forall\,u_1,u_2\in\mathbb{V}.
\end{align}

Set
\[
\widehat{F}(u):=(I+\alpha A)^{-1} F(u).
\]
Applying $(I+\alpha A)^{-1}$ to each term of the equation (\ref{1.a}), we can rewrite it in the following abstract form:
\begin{align}\label{Abstract}
\left\{
\begin{aligned}
& du(t)+\nu \widehat{A}u(t)dt+\widehat{B}\big(u(t),u(t)\big)dt=\widehat{F}\big(u(t)\big)dt+ \epsilon u(t) \circ dW(t),\quad t> 0\\
& u(0)=f .
\end{aligned}
\right.
\end{align}

Let $\{\mathcal{F}_t\}_{t\geq 0}$ be the augmented natural filtration of the Brownian motion $\{W(t): t\geq 0\}$.
\begin{dfn}\label{Def 01}
A $\mathbb{V}$-valued continuous and $\mathbb{W}$-valued weakly continuous $\{\mathcal{F}_t\}_{t\geq 0}$-adapted stochastic process $u$ is called a solution of the system (\ref{1.a}), if the following conditions hold:

\noindent (1) $u\in L^2(\Omega\times[0,T];\mathbb{W})$, for any $T>0$ .

\noindent (2) for any $t\geq 0$, the following equation holds in $\mathbb{W}^*$ $P$-a.s.:
\begin{align*}
u(t)+\nu\int_0^t\widehat{A}u(r)dr+\int_0^t \widehat{B}\big(u(r),u(r)\big)dr
=
f+\int_0^t\widehat{F}\big(u(r)\big)dr+\epsilon\int_0^t u(r) \circ  dW(r).
\end{align*}
\end{dfn}


With a minor modification of Theorem 3.2 in \cite{2017-Shang-p-}(or see \cite{2012-Razafimandimby-p4251-4270}), we have the following theorem.
\begin{thm}\label{thm existunique}
Assume {\bf(F1)} and $f\in\mathbb{W}$. Then there exists a unique solution to (\ref{Abstract}).
\end{thm}
In the deterministic case(i.e. $\epsilon=0$) of (\ref{Abstract}),
the following lemma is immediately deduced from Theorem 5.6 in \cite{1997-Cioranescu-p317-335} and Lemma
\uppercase\expandafter{\romannumeral3}.1.4 in \cite{1977-Temam-p500-500}.
\begin{lem}\label{Regularity}
Let $\epsilon=0$ in (\ref{Abstract}). Assume {\bf(F1)} and $f\in\mathbb{W}\cap \mathbb{H}^4(\mathcal{O})$. Then for any $T\geq 0$, the solution $u\in L^\infty\big([0,T];{\mathbb{H}}^4(\mathcal{O})\big)$. Moreover, $u(t)\in\mathbb{H}^4(\mathcal{O})$ for each $t\geq 0$, and
\begin{align}\label{Regularity estimate}
\sup_{t\in[0,T]}|u(t)|_{{\mathbb{H}}^4(\mathcal{O})}\leq C(T,|f|_{{\mathbb{H}}^4(\mathcal{O})}),\quad\forall\, T>0.
\end{align}
\end{lem}



We also introduce the following two conditions used in this paper.

\noindent {\bf(F2)} $F$ is $C^1$, i.e.
the Fr\'{e}chet derivative $\mathbb{D}F:\mathbb{V} \rightarrow L(\mathbb{V},\mathbb{V})$ is continuous, where $L(\mathbb{V},\mathbb{V})$ denotes the Banach space of all bounded linear operators from $\mathbb{V}$ to $\mathbb{V}$, for simplicity we write $L(\mathbb{V}):=L(\mathbb{V},\mathbb{V})$.

\noindent {\bf(F3)}  $F$ is Lipschitz continuous, i.e. $F$ satisfies (\ref{add 0321.1}). Furthermore, the Lipschitz constant $C_F<\frac{\nu}{\mathcal{P}^2}$, here $\mathcal{P}$ is the $\rm Poincar\acute{e}$ constant taken from (\ref{add 0329.1}), $\nu$ is the kinematic viscosity taken from (\ref{1.a}).

%

\vskip 0.3cm



Finally, we state our main results of this paper.

\begin{thm}\label{32.I}

Assume {\bf(F1)}. Then there is a version $u:\mathbb{R}^+\times\mathbb{W}\times\Omega\rightarrow\mathbb{W}$ of solutions to (\ref{Abstract}) that generates a continuous random dynamical system.

\end{thm}

\begin{thm}\label{32.II}
Assume {\bf(F1)} and {\bf(F2)}. Then for each $(t,\omega)\in\mathbb{R}^+\times\Omega$, the solution  $u(t,\cdot\,,\omega):\mathbb{W}\rightarrow\mathbb{V}$ in Theorem \ref{32.I} is continuously Fr\'{e}chet differentiable on the subset $\mathbb{W}\cap \mathbb{H}^4(\mathcal{O})$, where the `continuously' means that if $f\in\mathbb{W}\cap \mathbb{H}^4(\mathcal{O})$, $\{f_n\}_{n=1}^{\infty}\subset\mathbb{W}\cap \mathbb{H}^4(\mathcal{O})$ and $|f_n -f|_{\mathbb{W}}\rightarrow 0$, then the Fr\'{e}chet derivatives $\|\mathbb{D}u(t,f_n,\omega)-\mathbb{D}u(t,f,\omega)\|_{L(\mathbb{W},\mathbb{V})}\rightarrow 0$ as $n\rightarrow\infty$.
\end{thm}

\begin{thm}\label{32.III}
  Assume {\bf(F3)}. Then the random dynamical system in Theorem \ref{32.I} is asymptotically compact and possesses a unique random attractor $\mathcal{A}^{\epsilon}$. Moreover, the family of random attractors $\{\mathcal{A}^{\epsilon}\}_{|\epsilon|>0}$ is upper semi-continuous as the noise intensity $\epsilon\rightarrow 0$.

\end{thm}

\begin{remark}\label{add1}
Theorem \ref{32.II} reflects the fact that the solution $u(t,f,\omega)$ of (\ref{1.a}) has no smoothing effect in finite time.
Moreover, if we strengthen the regularity of $F$, then the map $u(t,\cdot\,,\omega):\mathbb{W}\rightarrow\mathbb{W}$ will be continuously Fr\'{e}chet differentiable on $\mathbb{W}\cap \mathbb{H}^5(\mathcal{O})$, the proof is similar to that of Theorem \ref{32.II}.

\end{remark}

\section{Proof of Theorem \ref{32.I}}
Our approach is to transform (\ref{Abstract}) into an equation (\ref{4.b})(see below) with random coefficients.
Let
\begin{align}
\label{add 0328.3} Q(t,\omega):&=\mathrm{e}^{\epsilon W(t,\omega)}, \\
\label{add 0406.1}\widehat{F}_Q(u,Q(t,\omega)):&=\widehat{F}\big(Q(t,\omega)u\big)/Q(t,\omega), \\
\label{add 0406.2}\widehat{\mathcal{A}}_Q (u,Q(t,\omega)):&= -\nu\widehat{A}u-Q(t,\omega)\widehat{B}(u,u)+\widehat{F}_Q\big(u,Q(t,\omega)\big).
\end{align}
For simplicity, we sometimes omit the parameter $\omega$ in the following when it is clear from the context.
Consider the following system for each fixed $\omega\in\Omega$,
\begin{align}\label{4.b}
\left\{
\begin{aligned}
& dv(t,f)=\widehat{\mathcal{A}}_Q (v(t,f),Q(t))\,dt,\quad t>0, \\
& v(0,f)=f,\quad \text{in}\ \mathbb{W}.
\end{aligned}
\right.
\end{align}
%
%
%
\noindent Once the existence and uniqueness of solutions to (\ref{4.b}) are established, by It\^{o}'s formula it is easy to see that
%
%
%
$Q(t)v(t,f)$ is a version of the solution to (\ref{Abstract}). Furthermore, we can show that $Q(t)v(t,f)$ generates a random dynamical system.



In the following of this section, we first establish the existence and uniqueness of solutions to (\ref{4.b}) by Galerkin approximations. Then we prove the continuity of the solution $v(t,f)$ with respect to $(t,f)\in \mathbb{R}^+\times\mathbb{W}$. Finally, we give the proof of Theorem \ref{32.I}.

From (\ref{Basis}), we know that $\{\sqrt{\lambda_i}e_i\}_{i=1}^{\infty}$ is an orthonormal basis of $\mathbb{V}$. $\mathbb{V}_n$ denotes the n-dimensional subspace spanned by $\{e_1,e_2,\ldots ...,e_n\}$ in $\mathbb{W}$. Let $\Pi_n:\mathbb{W}^*\rightarrow\mathbb{V}_n$ be defined by
\[ \Pi_ng:=\sum_{i=1}^n\lambda_i\langle g,e_i\rangle e_i,\quad \forall\, g\in\mathbb{W}^*. \]
Now for any integer $n\geq1$, we seek a solution $v_n\in\mathbb{V}_n$ to the equation
\begin{align}\label{5.a}
\left\{
\begin{aligned}
& dv_n(t)= \Pi_n\widehat{\mathcal{A}}_Q (v_n(t),Q(t))\,dt,\quad t>0, \\
& v_n(0)=\Pi_n f .
\end{aligned}
\right.
\end{align}



\begin{lem}\label{Lemma 6.I}
Assume {\bf(F1)} and $f\in\mathbb{W}$. Then there exists a unique global solution to (\ref{5.a}). Moreover, for any $\omega\in\Omega$ and $T\geq 0$, the following estimate holds:
\begin{align}
\label{W norm estimate of v_n} \sup_{n\in\mathbb{N}}\sup_{t\in[0,T]}\big|v_n(t)\big|_{\mathbb{W}}^2\leq C(T)\bigg(|f|_{\mathbb{W}}^2+|F(0)|_{\mathbb{V}}^2\int_0^T \frac{1}{Q(s)^2}\,ds\bigg) .
\end{align}
\end{lem}

\begin{proof}
Actually, $v_n(t)$ solves (\ref{5.a}) if and only if the Fourier coefficients of $v_n$ solve a system of random ordinary differential equations with locally Lipschitz coefficients.
%
Therefore, (\ref{5.a}) admits a unique local solution defined on the maximal existence time interval $\big[0,T_0(\omega)\big)$.
From (\ref{5.a}) it follows that
\begin{align}\label{6.add}
      d\big|v_n(t)\big|_{\mathbb{V}}^2=2 \big\langle \widehat{\mathcal{A}}_Q (v_n(t),Q(t)), v_n(t)\big\rangle\,dt.
\end{align}

\noindent By (\ref{inverse operator transform}), (\ref{A transform 01}), Lemma \ref{Lem-B-01} and {\bf(F1)}, we get
\begin{align*}
      \big|v_n(t)\big|_{\mathbb{V}}^2+2\nu\int_0^t\big\|v_n(s)\big\|^2\,ds\leq |f|_{\mathbb{V}}^2 + C|F(0)|_{\mathbb{V}}^2\int_0^{T_0}\frac{1}{Q(s)^2}\,ds + C\int_0^t \big|v_n(s)\big|_{\mathbb{V}}^2\,ds .
\end{align*}
Hence Gronwall's inequality implies
\begin{align}\label{7.a}
\sup_{0\leq t <T_0}\big|v_n(t)\big|_{\mathbb{V}}^2+2\nu\int_0^{T_0} \big\|v_n(t)\big\|^2\,dt\leq C(T_0)\bigg(|f|_{\mathbb{V}}^2 + |F(0)|_{\mathbb{V}}^2\int_0^{T_0}\frac{1}{Q(s)^2}\,ds\bigg) ,
\end{align}
which concludes that $v_n$ is global in time, i.e. $T_0=\infty$ for every $\omega\in\Omega$.

Next, we estimate the $\mathbb{W}$-norm of $v_n$.
By (\ref{5.a}), we have
\begin{align}\label{14.1}
\begin{aligned}
& d(v_n(t),e_i)_{\mathbb{V}}=(\widehat{\mathcal{A}}_Q (v_n(t),Q(t)),e_i)_{\mathbb{V}}\,dt,\quad i=1,\ldots,n.
\end{aligned}
\end{align}

\noindent Since $\widehat{\mathcal{A}}_Q (v_n(t),Q(t))\in\mathbb{W}$, multiplying both sides of (\ref{14.1}) by $\lambda_i$, we can use (\ref{Basis}) to obtain
\begin{align}\label{14.1.a}
d(v_n(t),e_i)_{\mathbb{W}}=(\widehat{\mathcal{A}}_Q (v_n(t),Q(t)),e_i)_{\mathbb{W}}\,dt,\quad i=1,\ldots,n.
\end{align}

\noindent Applying the chain rule to $(v_n(t),e_i)_{\mathbb{W}}^2$ and then summing over $i$ from $1$ to $n$ yields
\begin{align}\label{16.1}
\begin{aligned}
d|v_n(t)|_{\mathbb{W}}^2= 2(\widehat{\mathcal{A}}_Q (v_n(t),Q(t)),v_n(t))_{\mathbb{W}}\,dt.
\end{aligned}
\end{align}

\noindent Due to (\ref{Eq B-06}), (\ref{W equivalence norm}) and (\ref{curlP=curl}), we derive
\begin{align}\label{16.2}
\begin{aligned}
 &\big(\widehat{\mathcal{A}}_Q (v_n(t),Q(t)),v_n(t)\big)_{\mathbb{W}}\\
=& \big(\widehat{F}_Q(v_n(t),Q(t)),v_n(t)\big)_{\mathbb{W}}+ \nu\Big({\rm curl}(\Delta v_n(t)),{\rm curl}\big(v_n(t)-\alpha\Delta v_n(t)\big)\Big) \\
=& \big(\widehat{F}_Q(v_n(t),Q(t)),v_n(t)\big)_{\mathbb{W}}-\frac{\nu}{\alpha}\big|v_n(t)\big|_{\mathbb{W}}^2 + \frac{\nu}{\alpha}\Big({\rm curl}\big(v_n(t)\big),{\rm curl}\big(v_n(t)-\alpha\Delta v_n(t)\big)\Big) .
\end{aligned}
\end{align}

\noindent Substituting (\ref{16.2}) into (\ref{16.1}) gives
\begin{align}\label{17.1}
\begin{aligned}
& |v_n(t)|_{\mathbb{W}}^2+ \frac{2\nu}{\alpha}\int_0^t|v_n(s)|_{\mathbb{W}}^2\,ds \\
\leq & |\Pi_n f|_{\mathbb{W}}^2 + \frac{2\nu}{\alpha}\int_0^t \big|{\rm curl}\big(v_n(s)\big)\big|\big|v_n(s)\big|_{\mathbb{W}}\,ds + 2\int_0^t \big|\widehat{F}_Q(v_n(s),Q(s))\big|_{\mathbb{W}}\big|v_n(s)\big|_{\mathbb{W}}\,ds \\
\leq & |\Pi_n f|_{\mathbb{W}}^2 + \frac{2\nu}{\alpha}\int_0^t \big|{\rm curl}\big(v_n(s)\big)\big|^2\,ds + \frac{\nu}{\alpha}\int_0^t \big|v_n(s)\big|_{\mathbb{W}}^2 ds \\ &+ \frac{2\alpha}{\nu}\int_0^t \big|\widehat{F}_Q(v_n(s),Q(s))\big|_{\mathbb{W}}^2\,ds .
\end{aligned}
\end{align}

\noindent Hence
\begin{align}
\begin{aligned}
&\big|v_n(t)\big|_{\mathbb{W}}^2 + \frac{\nu}{\alpha}\int_0^t \big|v_n(s)\big|_{\mathbb{W}}^2\,ds \\ \leq &|\Pi_n f|_{\mathbb{W}}^2 + \frac{2\nu}{\alpha}\int_0^t \big|{\rm curl}\big(v_n(s)\big)\big|^2\,ds + \frac{2\alpha}{\nu}\int_0^t \big|\widehat{F}_Q(v_n(s),Q(s))\big|_{\mathbb{W}}^2\,ds.
\end{aligned}
\end{align}


\noindent By (\ref{curl and V}), (\ref{inverse operator transform inequality}), {\bf (F1)} and Gronwall's inequality, we deduce
%
\begin{align}\label{18.1.1}
\sup_{0\leq t\leq T}\big|v_n(t)\big|_{\mathbb{W}}^2 + \frac{\nu}{\alpha}\int_0^T \big|v_n(s)\big|_{\mathbb{W}}^2\,ds \leq C(T)\bigg(|\Pi_n f|_{\mathbb{W}}^2+|F(0)|_{\mathbb{V}}^2\int_0^T \frac{1}{Q(s)^2}\,ds\bigg) .
\end{align}
Obviously, $|\Pi_n f|_{\mathbb{W}}^2\leq |f|_{\mathbb{W}}^2$ for any $n\in\mathbb{N}$. Hence (\ref{W norm estimate of v_n}) follows from (\ref{18.1.1}).
\end{proof}

\begin{prp}\label{5.I}
Assume {\bf(F1)} and $f\in\mathbb{W}$. Then for any $\omega\in\Omega$ and $T>0$, there exists a unique solution $v(\cdot\,,f,\omega)\in C\big([0,T];\mathbb{V}\big)\cap L^\infty\big([0,T];{\mathbb{W}}\big)$ to (\ref{4.b}).
Moreover, the following estimate holds:
\begin{align}
\label{W norm estimate of v} &\sup_{t\in[0,T]}\big|v(t,f,\omega)\big|_{\mathbb{W}}^2\leq C(T)\bigg(|f|_{\mathbb{W}}^2+|F(0)|_{\mathbb{V}}^2\int_0^T \frac{1}{Q(s)^2}\,ds\bigg) .
\end{align}

\end{prp}


\begin{proof}
Fix $\omega\in\Omega$ and $T>0$.
Obviously, $\|Q\|_{\infty,T}:=\sup_{0\leq t\leq T}Q(t)<\infty$.
By (\ref{A transform 03}), (\ref{add 0408.1}), (\ref{inverse operator transform inequality}), {\bf{(F1)}} and (\ref{W norm estimate of v_n}), we have
\begin{align}\label{11.c}
\begin{aligned}
&\sup_{0\leq t\leq T} \left|v_n'(t)\right|_{\mathbb{V}}  \\
=&\sup_{0\leq t\leq T} \big|\Pi_n\widehat{\mathcal{A}}_Q (v_n(t),Q(t))\big|_{\mathbb{V}} \\
\leq & \sup_{0\leq t\leq T}\big|\widehat{A}v_n(t)\big|_{\mathbb{V}}  +
  \sup_{0\leq t\leq T}\big|Q(t)\widehat{B}\big(v_n(t)\big)\big|_{\mathbb{V}}   +
  \sup_{0\leq t\leq T}\big|\widehat{F}_Q\big(v_n(t),Q(t)\big)\big|_{\mathbb{V}} \\
\leq& C\sup_{0\leq t\leq T}\big|v_n(t)\big|_{\mathbb{V}}  +
  C\sup_{0\leq t\leq T}|Q(t)|\sup_{0\leq t\leq T}\big|v_n(t)\big|_{\mathbb{W}}^2
  +C\sup_{0\leq t\leq T} \frac{|F(0)|_{\mathbb{V}}}{Q(s)}  \\
  \leq & C(T,|f|_{\mathbb{W}},|F(0)|_{\mathbb{V}},\omega), \quad \forall\, n\in\mathbb{N}.
\end{aligned}
\end{align}


\noindent (\ref{11.c}) and (\ref{W norm estimate of v_n}) imply that there exist a subsequence $\{v_{n_k}\}_{k=1}^{\infty}$ and $v$ such that:
\begin{align}
\label{12.a.b}& v_{n_k}\rightharpoonup v \,\quad\text{weakly star in}\  L^{\infty}\big([0,T];\mathbb{W}\big) , \\
\label{12.a.d}& v_{n_k}'\rightharpoonup v' \quad\text{weakly star in}\  L^{\infty}\big([0,T];\mathbb{V}\big) .
\end{align}

\noindent Moreover, $v$ satisfies the estimate (\ref{W norm estimate of v}), and $v'(\cdot\, ,f,\omega)\in L^{\infty}\big([0,T];\mathbb{V}\big)$.
Using Lemma
\uppercase\expandafter{\romannumeral3}.1.2 in \cite{1977-Temam-p500-500}, we deduce that $v(\cdot\, ,f,\omega)\in C\big([0,T];\mathbb{V}\big)$ for any $f\in\mathbb{W}$.
Lemma
\uppercase\expandafter{\romannumeral3}.1.4 in \cite{1977-Temam-p500-500} implies that $v(t,f,\omega)\in\mathbb{W}$ for any $t\geq 0$, and $v(\cdot\, ,f,\omega)$ is $\mathbb{W}$-valued weakly continuous.
Since the embedding $\mathbb{W}\hookrightarrow\mathbb{V}$ is compact, by Theorem \uppercase\expandafter{\romannumeral3}.2.1 in \cite{1977-Temam-p500-500}, it follows that
\begin{align}\label{vn strongly convergent in V}
v_{n_k}\rightarrow v \quad \text{strongly in}\ L^2\big([0,T];\mathbb{V}\big).
\end{align}

\noindent By (\ref{12.a.b}), (\ref{vn strongly convergent in V}), the bilinear property of $\widehat{B}$ and the following estimate
\begin{align}\label{B(u,v) estimate}
  \begin{aligned}
    \|\widehat{B}(v(\cdot),v(\cdot))\|_{L^1([0,T];\mathbb{W}^*)}\leq C\|v(\cdot)\|_{L^2([0,T];\mathbb{W})}\|v(\cdot)\|_{L^2([0,T];\mathbb{V})},
  \end{aligned}
\end{align}

\noindent we obtain that
\begin{align}\label{B(vn,vn) limit }
\widehat{B}\big(v_{n_k}(s)\big)\rightharpoonup \widehat{B}\big(v(s)\big)\,\quad\text{weakly in}\   L^1([0,T];\mathbb{W}^*).
\end{align}

\noindent Therefore, we can pass to the limit in (\ref{5.a}) to conclude that $v$ is a solution of (\ref{4.b}).

%
%
%
%

To prove the uniqueness, we consider the following equation:
\begin{align}\label{13.a}
\left\{
\begin{aligned}
& d\big(v(t,f)-v(t,g)\big)=\big(\widehat{\mathcal{A}}_Q (v(t,f),Q(t))-\widehat{\mathcal{A}}_Q (v(t,g),Q(t))\big)\,dt,\quad t>0,\\
& v(0,f)-v(0,g)=f-g .
\end{aligned}
\right.
\end{align}

\noindent where $f,g\in\mathbb{W}$. By the chain rule and Lemma \ref{Lem-B-01}, we have
\begin{align*}
& \big|v(t,f)-v(t,g)\big|_{\mathbb{V}}^2 \\
=&\,|f-g|_{\mathbb{V}}^2-2\nu\int_0^t\big\langle\widehat{\mathcal{A}}_Q (v(s,f),Q(s))-\widehat{\mathcal{A}}_Q (v(s,g),Q(s)),v(s,f)-v(s,g)\big\rangle\,ds \\
=&\,|f-g|_{\mathbb{V}}^2- 2\nu\int_0^t\big\|v(s,f)-v(s,g)\big\|^2\,ds
+ 2\int_0^t Q(s)\big\langle\widehat{B}\big(v(s,f)-v(s,g)\big),v(s,g)\big\rangle\,ds \\ &+2\int_0^t\big\langle\widehat{F}_Q\big(v(s,f),Q(s)\big)-\widehat{F}_Q\big(v(s,g),Q(s)\big),v(s,f)-v(s,g)\big\rangle\,ds\\
\leq & |f-g|_{\mathbb{V}}^2- 2\nu\int_0^t\big\|v(s,f)-v(s,g)\big\|^2\,ds  +C\int_0^t Q(s)\big|v(s,f)-v(s,g)\big|_{\mathbb{V}}^2\big|v(s,g)\big|_{\mathbb{W}}\,ds\\
& +C\int_0^t\big|v(s,f)-v(s,g)\big|_{\mathbb{V}}^2\,ds.
\end{align*}
Applying Gronwall's inequality and using (\ref{W norm estimate of v}), we obtain
\begin{align}\label{14.a}
\begin{aligned}
& \sup_{0\leq t\leq T}\big|v(s,f)-v(s,g)\big|_{\mathbb{V}}^2 + 2\nu\int_0^T\big\|v(s,f)-v(s,g)\big\|^2\,ds \\
\leq &\,|f-g|_{\mathbb{V}}^2\exp\Big(CT+C\int_0^T Q(s)\big|v(s,g)\big|_{\mathbb{W}}\,ds\Big)\\
\leq &\,|f-g|_{\mathbb{V}}^2\exp\Big(CT+\|Q\|_{\infty,T}C(|g|_{\mathbb{W}},T,\omega)\Big) ,
\end{aligned}
\end{align}

\noindent which implies the uniqueness.
\end{proof}

The following energy equation will be used several times in this paper, so we give it here.
\begin{lem}\label{add 0323 lemma.1}
Assume {\bf(F1)} and $f\in\mathbb{W}$. Then the solution $v(t,f)$ of (\ref{4.b}) satisfies the following energy equation:
\begin{align}\label{energy equation}
   |v(t,f)|_{\mathbb{W}}^2=|f|_{\mathbb{W}}^2\mathrm{e}^{-\frac{2\nu}{\alpha}t}+2\int_0^{t}K(v(s,f),Q(s))\mathrm{e}^{-\frac{2\nu}{\alpha}(t-s)}\,ds, \quad \forall\, t\geq 0 ,
\end{align}
where
\begin{align}\label{add 0323.2}
\begin{aligned}
  &K(v(s,f),Q(s))\\ :=&\bigg(\frac{\nu}{\alpha}{\rm curl}\Big(v(s,f)\Big)+{\rm curl}\Big(Q(s)^{-1}F\big(Q(s)v(s,f)\big)\Big),{\rm curl}\Big(v(s,f)-\alpha\Delta v(s,f)\Big)\bigg).
\end{aligned}
\end{align}

\end{lem}

\begin{proof}
We only give the idea of the proof, since the details are similar to the proof of Theorem 4.1.2 in \cite{1998-Moise-p1369-1393}.
From (\ref{16.1}) and (\ref{16.2}), it follows that $v_n(t,f_n)$ satisfies this energy equation. Then let $n\rightarrow\infty$. Due to the weak convergence of $v_n(t,f_n)$, we can obtain that the left hand side of (\ref{energy equation}) is less than or equal to the right hand side of (\ref{energy equation}). The opposite inequality can be proved using the time reversing technique.

\end{proof}


\begin{prp}\label{16.6.I}
Assume {\bf(F1)}. Then for any $\omega\in\Omega$, the map $\mathbb{R}^{+}\times\mathbb{W}\ni (t,f)\mapsto v(t,f,\omega)\in\mathbb{W}$ is continuous.
\end{prp}
\begin{proof}
Fix $t\geq 0$ and $f\in\mathbb{W}$. Let $\{t_n\}_{n=1}^{\infty}$ be any sequence in $\mathbb{R}^{+}$ such that $t_n\rightarrow t$. So there exists a constant $T> 0$ such that $\sup_n t_n< T$.
Let $\{f_n\}_{n=1}^{\infty}$ be any sequence in $\mathbb{W}$ such that $f_n\rightarrow f$ strongly in $\mathbb{W}$. So there exists a constant $c> 0$ such that $\sup_n |f_n|_{\mathbb{W}}^2< c$.
Our aim is to show that $v(t_n,f_n)\rightarrow v(t,f)$ strongly in $\mathbb{W}$. We prove this with the following three steps.

Step 1.
By the same method as the proof of Proposition \ref{5.I}, we can extract a subsequence $\{v(\cdot,f_{n_j})\}_{j=1}^{\infty}$ such that
\begin{align}
\label{continuous convergence1}& v(\cdot,f_{n_j})\rightharpoonup \widetilde{v}(\cdot) \,~\quad\text{weakly star in}\  L^{\infty}\big([0,T];\mathbb{W}\big) , \\
\label{continuous convergence3}& v'(\cdot,f_{n_j})\rightharpoonup \widetilde{v}'(\cdot) \quad\text{weakly star in}\  L^{\infty}\big([0,T];\mathbb{V}\big),\\
\label{continuous convergence4}& v(\cdot,f_{n_j})\rightarrow \widetilde{v}(\cdot) \,~\quad\text{strongly in}\  L^2\big([0,T];\mathbb{V}\big) ,
\end{align}
for some $\widetilde{v}\in L^{\infty}\big([0,T];\mathbb{W}\big)$ with $\widetilde{v}'\in L^{\infty}\big([0,T];\mathbb{V}\big)$. (\ref{continuous convergence1})-(\ref{continuous convergence4}) allow us to pass to the limit in the equation for $v(\cdot,f_{n_j},\omega)$ to find that $\widetilde{v}(\cdot)$ is also a solution of (\ref{4.b}) with $\widetilde{v}(0)=f$.
Then by the uniqueness of solutions to (\ref{4.b}), we obtain $\widetilde{v}(\cdot)=v(\cdot,f)$.
Moreover, by a contradiction argument, the uniqueness of solutions to (\ref{4.b}) implies that the whole sequence $\{v(\cdot,f_n)\}_{n=1}^{\infty}$ converges to $v(\cdot,f)$ in the sense of (\ref{continuous convergence1})-(\ref{continuous convergence4}).

Step 2.
Since
\begin{align}
 v(\cdot,f_{n})\rightarrow v(\cdot,f) \,\quad\text{strongly in}\  L^2\big([0,T];\mathbb{V}\big),
\end{align}
there exists a subsequence $\{n_k\}_{k=1}^{\infty}$ and a subset $\Lambda\subset [0,T]$ of Lebesgue measure zero such that
\begin{align}\label{add 0409.1}
  v(s,f_{n_k})\rightarrow v(s,f),\quad \text{ for every } s\in [0,T]\backslash\Lambda.
\end{align}
Take any sequence $\{s_m\}_{m=1}^{\infty}\subset [0,T]\backslash\Lambda$ such that $s_m\rightarrow t$.
Similar to (\ref{11.c}), we have
\begin{align}
  \sup_{n\in\mathbb{N}}\sup_{0\leq r\leq T} |v'(r,f_n)|_{\mathbb{V}}\leq C(T,c,|F(0)|_{\mathbb{V}},\omega).
\end{align}
Hence for any $a,b\in [0,T]$ and $a<b$,
\begin{align}\label{add 0409.2}
  \sup_{n\in\mathbb{N}}|v(a,f_n)-v(b,f_n)|_{\mathbb{V}}\leq \sup_{n\in\mathbb{N}}\int_a^b |v'(r,f_n)|_{\mathbb{V}}dr \leq C(T,c,|F(0)|_{\mathbb{V}},\omega)(b-a).
\end{align}
Note
\begin{align}\label{add 0416.1}
\begin{aligned}
  v(t_{n_k},f_{n_k})-v(t,f)=&v(t_{n_k},f_{n_k})-v(t,f_{n_k})+v(t,f_{n_k})-v(s_m,f_{n_k}) \\
  &+v(s_m,f_{n_k})-v(s_m,f)+v(s_m,f)-v(t,f) .
\end{aligned}
\end{align}
First let $k\rightarrow\infty$, and then let $m\rightarrow\infty$ in (\ref{add 0416.1}), together with (\ref{add 0409.1}) and (\ref{add 0409.2}), we obtain
\begin{align}
  |v(t_{n_k},f_{n_k})-v(t,f)|_{\mathbb{V}}\rightarrow 0 ,\quad \text{ as } k\rightarrow\infty.
\end{align}
Again, by a contradiction argument, the whole sequence $\{v(t_n,f_n)\}_{n=1}^{\infty}$ converges to $v(t,f)$ strongly in $\mathbb{V}$.
Since $\sup_{n}|v(t_n,f_n)|_{\mathbb{W}}<\infty$ by (\ref{W norm estimate of v}), and $\mathbb{V}$ is dense in $\mathbb{W}^*$, we get
%
\begin{align}\label{continuous weak W}
  v(t_n,f_n)\rightharpoonup v(t,f) \quad\text{weakly in }\mathbb{W} \text{ as } n\rightarrow\infty.
\end{align}

Step 3.
By Lemma \ref{add 0323 lemma.1}, we have the following energy equation for $v(t_n,f_n)$:
\begin{align}\label{energy equation v_n}
  |v(t_n,f_n)|_{\mathbb{W}}^2=|f_n|_{\mathbb{W}}^2\mathrm{e}^{-\frac{2\nu}{\alpha}t_n}+2\int_0^{t_n}K(v(s,f_n),Q(s))\mathrm{e}^{-\frac{2\nu}{\alpha}(t_n-s)}\,ds,
\end{align}
Letting $n\rightarrow\infty$ in (\ref{energy equation v_n}), together with (\ref{continuous convergence1}) and (\ref{continuous convergence4}), we obtain
\begin{align}
  \lim_{n\rightarrow\infty}|v(t_n,f_n)|_{\mathbb{W}}^2=|f|_{\mathbb{W}}^2\mathrm{e}^{-\frac{2\nu}{\alpha}t}+2\int_0^{t}K(v(s,f),Q(s))\mathrm{e}^{-\frac{2\nu}{\alpha}(t-s)}\,ds.
\end{align}
From the energy equation for $v(t,f)$, we see that
\begin{align}
  \lim_{n\rightarrow\infty}|v(t_n,f_n)|_{\mathbb{W}}^2=|v(t,f)|_{\mathbb{W}}^2,
\end{align}
which together with (\ref{continuous weak W}) yields
\begin{align}
  v(t_n,f_n)\rightarrow v(t,f) \quad\text{strongly in }\mathbb{W}  \text{ as } n\rightarrow\infty.
\end{align}

\end{proof}

\begin{remark}\label{add 0320 remark}
Assume {\bf(F1)}. Then for any $\omega\in\Omega$ and $t\geq 0$, the map $\mathbb{W}\ni f\mapsto v(t,f,\omega)\in\mathbb{W}$ is weakly continuous, i.e. if $f_n$ converges weakly to $f$ in $\mathbb{W}$, then $v(t,f_n,\omega)$ converges weakly to $v(t,f,\omega)$ in $\mathbb{W}$ as $n\rightarrow\infty$. This is followed from (\ref{continuous weak W}).

\end{remark}

We are now in the position to give the proof of Theorem \ref{32.I}.
\begin{proof}[{\bf Proof of Theorem \ref{32.I}}]

Let $v(t,f,\omega)$ be the solution of (\ref{4.b}). Define the required version $u:\mathbb{R}^+\times\mathbb{W}\times\Omega\rightarrow\mathbb{W}$ by
\begin{align}\label{34.a}
 u(t,f,\omega):=Q(t,w)v(t,f,\omega).
\end{align}
By It\^{o}'s formula, it is easy to see that this process $u$ is a solution of (\ref{Abstract}).
Using the similar arguments as the proof of Theorem 3.2(iii) in \cite{2010-Mohammed-p3543-3591}, we can verify the cocycle property of $u$. The continuity of the cocycle follows from (\ref{34.a}) and Proposition \ref{16.6.I}.

\end{proof}

\section{Proof of Theorem \ref{32.II}}


To prove Theorem \ref{32.II}, it suffices to show the Fr\'{e}chet differentiability of $v(t,f,\omega)$ in view of (\ref{34.a}). In this section, we fix $\omega\in\Omega$.
We first give the following estimate of the solution $v(t,f)$ to (\ref{4.b}).
\begin{prp}\label{15.1.I}
Assume {\bf(F1)}, $f\in\mathbb{W}\cap \mathbb{H}^4(\mathcal{O})$ and $g\in\mathbb{W}$. Then for any $\omega\in\Omega$ and $T\geq 0$,
\begin{align}\label{15.1.a}
\sup_{0\leq t\leq T}\big|v(t,f)-v(t,g)|_{\mathbb{W}}^2\leq C(|f|_{\mathbb{H}^4(\mathcal{O})},T,\omega)|f-g|_{\mathbb{W}}^2.
\end{align}
\end{prp}
\begin{proof}
Since the subset $\mathbb{W}\cap \mathbb{H}^4(\mathcal{O})$ is dense in $\mathbb{W}$, by Proposition \ref{16.6.I}, it suffices to assume $g\in\mathbb{W}\cap \mathbb{H}^4(\mathcal{O})$ in the following proof. Consider the equation (\ref{13.a}) and let $x(t):=v(t,f)-v(t,g)$.
According to Lemma \ref{Regularity}, we see that $\widehat{\mathcal{A}}_Q (v(t,f),Q(t)), \widehat{\mathcal{A}}_Q (v(t,g),Q(t))\in\mathbb{W}$.
%
Using the similar argument as for (\ref{16.1}), we have
\begin{align}\label{15.2.a}
\begin{aligned}
|x(t)|_{\mathbb{W}}^2= |x(0)|_{\mathbb{W}}^2 +2\int_0^t\big(\widehat{\mathcal{A}}_Q (v(s,f),Q(s))-\widehat{\mathcal{A}}_Q (v(s,g),Q(s)),x(s)\big)_{\mathbb{W}}\,ds.
\end{aligned}
\end{align}

\noindent Due to (\ref{Eq B-06}), (\ref{W equivalence norm}), (\ref{curlP=curl}) and (\ref{Eq B-05}), we derive
\begin{align}\label{15.2.b}
& \big(\widehat{\mathcal{A}}_Q (v(s,f),Q(s))-\widehat{\mathcal{A}}_Q (v(s,g),Q(s)),x(s)\big)_{\mathbb{W}}\nonumber\\
=&\, \big(\widehat{F}_Q(v(s,f),Q(s))-\widehat{F}_Q(v(s,g),Q(s)),x(s)\big)_{\mathbb{W}}+ \nu\Big({\rm curl}\big(\Delta x(s)\big),{\rm curl}\big(x(s)-\alpha\Delta x(s)\big)\Big) \nonumber\\
&\,-Q(s)\big(\widehat{B}(v(s,f)),x(s)\big)_{\mathbb{W}}+Q(s)\big(\widehat{B}(v(s,g)),x(s)\big)_{\mathbb{W}}\nonumber\\
=&\, \big(\widehat{F}_Q(v(s,f),Q(s))-\widehat{F}_Q(v(s,g),Q(s)),x(s)\big)_{\mathbb{W}}\nonumber\\ &\,-\frac{\nu}{\alpha}\big|x(s)\big|_{\mathbb{W}}^2 + \frac{\nu}{\alpha}\Big({\rm curl}\big(x(s)\big),{\rm curl}\big(x(s)-\alpha\Delta x(s)\big)\Big) \nonumber\\
&\,-Q(s)\Big(\big(x(s)\cdot\nabla\big){\rm curl}\big(v(s,f)-\alpha\Delta v(s,f)\big),{\rm curl}\big(x(s)-\alpha\Delta x(s)\big)\Big).
\end{align}
Substituting the above equality into (\ref{15.2.a}) gives
\begin{align}
\begin{aligned}
& |x(t)|_{\mathbb{W}}^2+ \frac{2\nu}{\alpha}\int_0^t|x(s)|_{\mathbb{W}}^2\,ds \\
\leq & |f-g|_{\mathbb{W}}^2 + \frac{2\nu}{\alpha}\int_0^t \big|{\rm curl}\big(x(s)\big)\big|\big|x(s)\big|_{\mathbb{W}}\,ds \\ &+ 2\int_0^t \big|\widehat{F}_Q(v(s,f),Q(s))-\widehat{F}_Q(v(s,g),Q(s))\big|_{\mathbb{W}}\big|x(s)\big|_{\mathbb{W}}\,ds \\
& +C\int_0^t Q(s)\big|x(s)\big|_{L^{\infty}(\mathcal{O})}\big|v(s,f)\big|_{\mathbb{H}^{4}(\mathcal{O})}\big|x(s)\big|_{\mathbb{W}}\,ds.
\end{aligned}
\end{align}

\noindent By $2ab\leq \epsilon a^2+ \epsilon^{-1}b^2$, (\ref{curl and V}), {\bf(F1)} and $|x(s)|_{L^{\infty}(\mathcal{O})}\leq C|x(s)|_{\mathbb{W}}$, there holds
\begin{align}\label{15.2.d}
|x(t)|_{\mathbb{W}}^2+ \frac{\nu}{\alpha}\int_0^t|x(s)|_{\mathbb{W}}^2\,ds
\leq |f-g|_{\mathbb{W}}^2 + C\int_0^t \big(1+Q(s)\big|v(s,f)\big|_{\mathbb{H}^{4}(\mathcal{O})}\big)\big|x(s)\big|_{\mathbb{W}}^2\,ds.
\end{align}

\noindent According to Lemma \ref{Regularity}, we see that
\begin{align}\label{15.3.a}
\big|v(\cdot\,,f)\big|_{L^{\infty}([0,T];\mathbb{H}^4(\mathcal{O}))}\leq C(|f|_{\mathbb{H}^4(\mathcal{O})},T,\omega).
\end{align}

\noindent Therefore, applying Gronwall's inequality to (\ref{15.2.d}), we get
\begin{align}
|x(t)|_{\mathbb{W}}^2+ \frac{\nu}{\alpha}\int_0^t|x(s)|_{\mathbb{W}}^2\,ds
\leq  C (|f|_{\mathbb{H}^4(\mathcal{O})},T,\omega)|f-g|_{\mathbb{W}}^2.
\end{align}
The proof of Proposition \ref{15.1.I} is complete.
\end{proof}


Let $f\in\mathbb{W}\cap \mathbb{H}^4(\mathcal{O})$ and $g\in\mathbb{W}$. Consider the following random equation
\begin{align}\label{17.a}
\left\{
\begin{aligned}
& d z(t,f)(g)
=-\nu \widehat{A}z(t,f)(g)\,dt-Q(t)\widehat{B}\big(z(t,f)(g),v(t,f)\big)\,dt \\ &~~~~~~~~~~~~~~~~~\,-Q(t)\widehat{B}\big(v(t,f),z(t,f)(g)\big)\,dt+\mathbb{D}\widehat{F}\big(Q(t)v(t,f)\big)z(t,f)(g)\,dt ,\quad t>0,\\
&z(0,f)(g)=g ,
\end{aligned}
\right.
\end{align}
where $\mathbb{D}\widehat{F}$ is the Fr\'{e}chet derivative of the map $\widehat{F}:\mathbb{V}\rightarrow\mathbb{V}$. By the definition of Fr\'{e}chet derivatives, $\mathbb{D}\widehat{F}(u)v=(I+\alpha A)^{-1}(\mathbb{D}F(u)v)$.
We will prove $z(t,f,\omega)=\mathbb{D}v(t,f,\omega)$, see Step 1 in the proof of Theorem \ref{32.II}. Before this, we need the following proposition.
\begin{prp}\label{Lemma 18.I}
Assume {\bf(F1)}, {\bf(F2)}, $f\in\mathbb{W}\cap \mathbb{H}^4(\mathcal{O})$ and $g\in\mathbb{W}$. Then for any $\omega\in\Omega$, there exists a unique solution $z(\cdot\, ,f,\omega)(g)$ to (\ref{17.a}),
and the following estimates hold
\begin{align}\label{V and W norm estimate of z}
\begin{aligned}
&\sup_{0\leq t\leq T}\big|z(t,f,\omega)(g)\big|_{\mathbb{V}}^2\leq C(|f|_{\mathbb{W}},T,\omega)|g|_{\mathbb{V}}^2, \quad\forall\, T>0,\\
& \sup_{0\leq t\leq T}\big|z(t,f,\omega)(g)\big|_{\mathbb{W}}^2\leq C(|f|_{\mathbb{H}^4(\mathcal{O})},T,\omega)|g|_{\mathbb{W}}^2, \quad\forall\, T>0.
\end{aligned}
\end{align}

\end{prp}

\begin{proof}
We omit the details, since the proof is similar to that of Proposition \ref{5.I} and Proposition \ref{15.1.I} just with a little adaptation.

\end{proof}

\begin{proof}[Proof of Theorem \ref{32.II}]
As pointed out at the beginning of this section, by (\ref{34.a}), we only need to show that
for any $\omega\in\Omega$ and $t\geq 0$, $v(t,\cdot,\omega): \mathbb{W}\ni f\mapsto v(t,f,\omega)\in\mathbb{V}$ is continuously Fr\'{e}chet differentiable on $\mathbb{W}\cap \mathbb{H}^4(\mathcal{O})$.

The proof is divided into two steps. In step 1, we prove the Fr\'{e}chet differentiability of $v$ and $\mathbb{D}v(t,f,\omega)=z(t,f,\omega)$. In step 2, we prove the continuity of the Fr\'{e}chet derivatives.

Step 1. Fix $T>0$, $\omega\in\Omega$, and $f\in\mathbb{W}\cap \mathbb{H}^4(\mathcal{O})$. Let $g\in\mathbb{W}$ and $h\in\mathbb{R}\verb|\|\{0\}$.
Set
\begin{align}
  X(s,f,g,h):=\frac{v(s,f+hg)-v(s,f)}{h}-z(s,f)(g) ,\quad s\in[0,T].
\end{align}

\noindent It follows from (\ref{4.b}) and (\ref{17.a}) that
\begin{align}\label{24.a}
 dX(s,f,g,h)=-\nu\widehat{A}X(s,f,g,h)\,ds -Q(s)\Upsilon(s)\,ds +\Gamma(s)\,ds,
\end{align}

\noindent where
\begin{align*}
&\begin{aligned}
  \Upsilon(s):=&\frac{-\widehat{B}\big(v(s,f+hg)\big) +\widehat{B}\big(v(s,f)\big)}{h} +\widehat{B}\big(z(s,f)(g),v(s,f)\big) +\widehat{B}\big(v(s,f),z(s,f)(g)\big),
&\end{aligned}
\\
&\begin{aligned}
  \Gamma(s):=&\frac{\widehat{F}_Q\big(v(s,f+hg),Q(s)\big)-\widehat{F}_Q\big(v(s,f),Q(s)\big)}{h} -\mathbb{D}\widehat{F}\big(Q(s)v(s,f)\big)z(s,f)(g).
&\end{aligned}
\end{align*}

\noindent Applying the chain rule to (\ref{24.a}) gives
\begin{align}\label{add 0414.1}
   \big|X(t,f,g,h)\big|_{\mathbb{V}}^2 =& -2\nu\int_0^t\big\|X(s,f,g,h)\big\|^2\,ds  -2\int_0^t Q(s)\big\langle\Upsilon(s),X(s,f,g,h)\big\rangle\,ds \nonumber\\
& +2\int_0^t \Big\langle\frac{\widehat{F}_Q\big(v(s,f+hg),Q(s)\big)-\widehat{F}_Q\big(v(s,f),Q(s)\big)}{h} \nonumber\\
& ~~~~~~~~~~~~-\mathbb{D}\widehat{F}\big(Q(s)v(s,f)\big)\frac{v(s,f+hg)-v(s,f)}{h},X(s,f,g,h)\Big\rangle\,ds \nonumber\\
& +2\int_0^t\big\langle\mathbb{D}\widehat{F}\big(Q(s)v(s,f)\big)X(s,f,g,h),X(s,f,g,h)\big\rangle\,ds .
\end{align}

\noindent Due to Lemma \ref{Lem-B-01}, by a simple calculation, we have
\begin{align}
\begin{aligned}\label{24.b}
  \big\langle\Upsilon(s),X(s,f,g,h)\big\rangle   =& -\big\langle\widehat{B}\big(X(s,f,g,h),v(s,f)\big),X(s,f,g,h)\big\rangle \\
& +h\big\langle\widehat{B}\big(z(s,f)(g),X(s,f,g,h)\big),z(s,f)(g)\big\rangle \\
& +h\big\langle\widehat{B}\big(X(s,f,g,h)\big),z(s,f)(g)\big\rangle,
\end{aligned}
\end{align}

\noindent From (\ref{add 0414.1}), (\ref{24.b}) and Lemma \ref{Lem-B-01}, we obtain
\begin{align}\label{25.b}
& \big|X(t,f,g,h)\big|_{\mathbb{V}}^2 +2\nu\int_0^t\big\|X(s,f,g,h)\big\|^2\,ds \nonumber\\
\leq & \ C\int_0^t Q(s)\big|X(s,f,g,h)\big|_{\mathbb{V}}^2\big|v(s,f)\big|_{\mathbb{W}}\,ds  +C\int_0^t Q(s)h\big|z(s,f)(g)\big|_{\mathbb{W}}^2\big|X(s,f,g,h)\big|_{\mathbb{V}}\,ds \nonumber\\
& +C\int_0^t Q(s)h\big|X(s,f,g,h)\big|_{\mathbb{V}}^2\big|z(s,f)(g)\big|_{\mathbb{W}}\,ds +\Phi(t)  +\int_0^t\big|X(s,f,g,h)\big|_{\mathbb{V}}^2\,ds \nonumber\\ &+2\int_0^t\big\|\mathbb{D}\widehat{F}\big(Q(s)v(s,f)\big)\big\|_{L(\mathbb{V})} \big|X(s,f,g,h)\big|_{\mathbb{V}}^2\,ds ,
\end{align}

\noindent where
\begin{align}
\Phi(t)
:=\int_0^t\bigg|&\frac{\widehat{F}_Q\big(v(s,f+hg),Q(s)\big)-\widehat{F}_Q\big(v(s,f),Q(s)\big)}{h} \nonumber\\
 &-\mathbb{D}\widehat{F}\big(Q(s)v(s,f)\big)\frac{v(s,f+hg)-v(s,f)}{h}\bigg|_{\mathbb{V}}^2\,ds .
\end{align}

\noindent {\bf(F1)} and {\bf(F2)} imply that
\begin{align}\label{DF is bounded}
\big\|\mathbb{D}F\big(Q(s)v(s,f)\big)\big\|_{L(\mathbb{V})}\leq C, \quad \forall\, s\in[0,T]  .
\end{align}
\noindent Note $h|z(s,f)(g)|_{\mathbb{W}}^2|X(s,f,g,h)|_{\mathbb{V}} \leq \frac{1}{2}h^2|z(s,f)(g)|_{\mathbb{W}}^2 +\frac{1}{2}|z(s,f)(g)|_{\mathbb{W}}^2|X(s,f,g,h)|_{\mathbb{V}}^2$. Applying Gronwall's inequality to (\ref{25.b}), we get
\begin{align}\label{26.b}
\begin{aligned}
& \sup_{0\leq t\leq T}\big|X(t,f,g,h)\big|_{\mathbb{V}}^2 +2\nu\int_0^T\big\|X(s,f,g,h)\big\|^2\,ds \\
\leq & \Big(Ch^2\int_0^T Q(s)\big|z(s,f)(g)\big|_{\mathbb{W}}^2\,ds +\Phi(T)\Big)
\times\exp\Big(C\int_0^T Q(s)\big|v(s,f)\big|_{\mathbb{W}}\,ds \\
&~~+C\int_0^T Q(s)\big|z(s,f)(g)\big|_{\mathbb{W}}^2\,ds
+Ch\int_0^T Q(s)|z(s,f)(g)\big|_{\mathbb{W}}\,ds +CT\Big).
\end{aligned}
\end{align}

\noindent By (\ref{14.a}), we have
\begin{align}\label{add 0414.2}
 \lim_{h\rightarrow0}\sup_{|g|_{\mathbb{W}}\leq 1}\Phi(T)
\leq &  \lim_{h\rightarrow0}\sup_{|g|_{\mathbb{W}}\leq1}
\bigg\{\int_0^T\int_0^1\Big\|\mathbb{D}\widehat{F}\Big(Q(s)\big[v(s,f)+\rho\big(v(s,f+hg)-v(s,f)\big)\big]\Big) \nonumber\\ &\quad\quad\quad\quad\ \ -\mathbb{D}\widehat{F}\big(Q(s)v(s,f)\big)\Big\|_{L(\mathbb{V})}^2 \left|\frac{v(s,f+hg)-v(s,f)}{h}\right|_{\mathbb{V}}^2\,d\rho ds\bigg\} \nonumber\\
\leq & C(|f|_{\mathbb{W}},T,\omega)\lim_{h\rightarrow0}\sup_{|g|_{\mathbb{W}}\leq 1}
\bigg\{\int_0^T\int_0^1\Big\|\mathbb{D}\widehat{F}\Big(Q(s)\big[v(s,f)\nonumber\\
&\quad\quad +\rho\big(v(s,f+hg)-v(s,f)\big)\big]\Big) -\mathbb{D}\widehat{F}\big(Q(s)v(s,f),s\big)\Big\|_{L(\mathbb{V})}^2\,d\rho ds\bigg\} \nonumber\\
= & 0,
\end{align}
where we have used (\ref{DF is bounded}), {\bf(F2)} and the dominated convergence theorem in the last step.
Therefore, from (\ref{26.b}), (\ref{W norm estimate of v}), (\ref{V and W norm estimate of z}) and (\ref{add 0414.2}), we obtain
\begin{align}
\lim_{h\rightarrow0}\sup_{|g|_{\mathbb{W}}\leq 1} \left\{\sup_{0\leq t\leq T}\left|X(t,f,g,h)\right|_{\mathbb{V}}^2 +2\nu\int_0^T\left\|X(s,f,g,h)\right\|\,ds\right\}=0 .
\end{align}
In particular, for any $0\leq t\leq T$,
\begin{align}
\lim_{h\rightarrow0}\sup_{|g|_{\mathbb{W}}\leq 1} \left|\frac{v(t,f+hg)-v(t,f)}{h}-z(t,f)(g)\right|_{\mathbb{V}}=0.
\end{align}
Therefore, by the arbitrariness of $T>0$, we conclude that for any $\omega\in\Omega$ and $t\geq 0$, if $f\in\mathbb{W}\cap \mathbb{H}^4(\mathcal{O})$, then the map $\mathbb{W}\ni f\mapsto v(t,f,\omega)\in\mathbb{V}$ is Fr\'{e}chet differentiable at $f$, and the Fr\'{e}chet derivative $\mathbb{D}v(t,f,\omega)=z(t,f,\omega)$.

Step 2. We show that for any $\omega\in\Omega$ and $t\geq 0$, the map $\mathbb{W}\ni f\rightarrow z(t,f)\in L(\mathbb{W},\mathbb{V})$ is continuous on $\mathbb{W}\cap \mathbb{H}^4(\mathcal{O})$,
i.e. for any $f\in\mathbb{W}\cap \mathbb{H}^4(\mathcal{O})$, $\{f_n\}_{n=1}^{\infty}\subset\mathbb{W}\cap \mathbb{H}^4(\mathcal{O})$ and $f_n\rightarrow f$ strongly in $\mathbb{W}$, we will prove $\|z(t,f_n)-z(t,f)\|_{L(\mathbb{W},\mathbb{V})}\rightarrow 0$.
Let $g\in\mathbb{W}$ and $|g|_{\mathbb{W}}\leq 1$. By (\ref{17.a}), we have for any $T>0$ and $t\in[0,T]$,
\begin{align}\label{27.b}
&\ \left|z(t,f_n)(g)-z(t,f)(g)\right|_{\mathbb{V}}^2 \nonumber\\
=& -2\nu\int_0^t\left\|z(s,f_n)(g)-z(s,f)(g)\right\|^2\,ds \nonumber\\
& -2\int_0^t Q(s)\big\langle\widehat{B}\big(z(s,f_n)(g)-z(s,f)(g),v(s,f_n)\big),z(s,f_n)(g)-z(s,f)(g)\big\rangle\,ds \nonumber\\
& -2\int_0^t Q(s)\big\langle\widehat{B}\big(z(s,f)(g),v(s,f_n)-v(s,f)\big),z(s,f_n)(g)-z(s,f)(g)\big\rangle\,ds \nonumber\\
& -2\int_0^t Q(s)\big\langle\widehat{B}\big(v(s,f_n),z(s,f_n)(g)-z(s,f)(g)\big),z(s,f_n)(g)-z(s,f)(g)\big\rangle\,ds \nonumber\\
& -2\int_0^t Q(s)\big\langle\widehat{B}\big(v(s,f_n)-v(s,f),z(s,f)(g)\big),z(s,f_n)(g)-z(s,f)(g)\big\rangle\,ds \nonumber\\
& +2\int_0^t \Big(\mathbb{D}\widehat{F}\big(Q(s)v(s,f_n)\big)\big(z(s,f_n)(g)-z(s,f)(g)\big),z(s,f_n)(g)-z(s,f)(g)\Big)_{\mathbb{V}}\,ds \nonumber\\
& +2\int_0^t \Big(\big[\mathbb{D}\widehat{F}\big(Q(s)v(s,f_n)\big)-\mathbb{D}\widehat{F}\big(Q(s)v(s,f)\big)\big]z(s,f)(g),z(s,f_n)(g)-z(s,f)(g)\Big)_{\mathbb{V}}\,ds \nonumber\\
=& -2\nu\int_0^t\big\|z(s,f_n)(g)-z(s,f)(g)\big\|^2\,ds +I_1+I_2+I_3+I_4+I_5+I_6 .
\end{align}
Now we estimate terms $I_1,I_2,I_3,I_4,I_5,I_6$.
It follows from Lemma \ref{Lem-B-01} that
\begin{align}
&&&\begin{aligned}
|I_1|\leq C\int_0^t Q(s)\big|z(s,f_n)(g)-z(s,f)(g)\big|_{\mathbb{V}}^2\big|v(s,f_n)\big|_{\mathbb{W}}\,ds ,
&&&\end{aligned}
\\
&&&\begin{aligned}
|I_2|\leq  & C\int_0^t Q(s)\big|z(s,f)(g)\big|_{\mathbb{W}} \big|z(s,f_n)(g)-z(s,f)(g)\big|_{\mathbb{V}}\big|v(s,f_n)-v(s,f)\big|_{\mathbb{W}}\,ds \\
\leq & C\int_0^t Q(s)\big|z(s,f)(g)\big|_{\mathbb{W}}^2\big|z(s,f_n)(g)-z(s,f)(g)\big|_{\mathbb{V}}^2\,ds  \\ & +C\int_0^t Q(s)\big|v(s,f_n)-v(s,f)\big|_{\mathbb{W}}^2\,ds ,
&&&\end{aligned}
\end{align}
\noindent $I_3=0$, $|I_4|$ has the same estimate as $|I_2|$. Obviously,
\begin{align}
&&&\begin{aligned}
|I_5|\leq 2\int_0^t\left\|\mathbb{D}\widehat{F}\big(Q(s)v(s,f_n)\big)\right\|_{L(\mathbb{V})}
\left|z(s,f_n)(g)-z(s,f)(g)\right|_{\mathbb{V}}^2\,ds ,
&&&\end{aligned}
\\
&&&\begin{aligned}
|I_6|\leq & 2\int_0^t\left\|\mathbb{D}\widehat{F}\big(Q(s)v(s,f_n)\big)-\mathbb{D}\widehat{F}\big(Q(s)v(s,f)\big)\right\|_{L(\mathbb{V})} \\ &~~~~~~~~\times\left|z(s,f)(g)\right|_{\mathbb{V}}\left|z(s,f_n)(g)-z(s,f)(g)\right|_{\mathbb{V}}\,ds\\
\leq & \int_0^t\left\|\mathbb{D}\widehat{F}\big(Q(s)v(s,f_n)\big)-\mathbb{D}\widehat{F}\big(Q(s)v(s,f)\big)\right\|_{L(\mathbb{V})}^2 ds \\ &+\int_0^t \left|z(s,f)(g)\right|_{\mathbb{V}}^2\left|z(s,f_n)(g)-z(s,f)(g)\right|_{\mathbb{V}}^2 ds .
&&&\end{aligned}
\end{align}

\noindent Substituting the above estimates of $I_1$-$I_6$ into (\ref{27.b}) gives
\begin{align}
& \left|z(t,f_n)(g)-z(t,f)(g)\right|_{\mathbb{V}}^2+2\nu\int_0^t\left\|z(t,f_n)(g)-z(t,f)(g)\right\|^2\,ds \nonumber\\
\leq &\, C\int_0^t Q(s)\big|v(s,f_n)-v(s,f)\big|_{\mathbb{W}}^2\,ds  +\int_0^t\left\|\mathbb{D}\widehat{F}\big(Q(s)v(s,f_n)\big) -\mathbb{D}\widehat{F}\big(Q(s)v(s,f)\big)\right\|_{L(\mathbb{V})}^2\,ds \nonumber\\
& +\int_0^t\Big[CQ(s)\left|v(s,f_n)\right|_{\mathbb{W}} +CQ(s)\left|z(s,f)(g)\right|_{\mathbb{W}}^2 +2\left\|\mathbb{D}\widehat{F}\big(Q(s)v(s,f_n)\big)\right\|_{L(\mathbb{V})} \nonumber\\
&~~~~~~~~~+\left|z(s,f)(g)\right|_{\mathbb{V}}^2\Big]\left|z(s,f_n)(g)-z(s,f)(g)\right|_{\mathbb{V}}^2\,ds .
\end{align}

\noindent Applying Gronwall's inequality, and using (\ref{W norm estimate of v}), (\ref{15.1.a}), (\ref{V and W norm estimate of z}) and $|g|_{\mathbb{W}}\leq 1$, we have
\begin{align}
& \sup_{0\leq t\leq T}\left|z(t,f_n)(g)-z(t,f)(g)\right|_{\mathbb{V}}^2 +2\nu\int_0^T\left\|z(t,f_n)(g)-z(t,f)(g)\right\|^2\,ds \nonumber\\
\leq & \bigg(C\int_0^T Q(s)\big|v(s,f_n)-v(s,f)\big|_{\mathbb{W}}^2\,ds  \nonumber\\ &~~+\int_0^T\left\|\mathbb{D}\widehat{F}\big(Q(s)v(s,f_n)\big) -\mathbb{D}\widehat{F}\big(Q(s)v(s,f)\big)\right\|_{L(\mathbb{V})}^2\,ds \bigg) \nonumber\\
&\times\exp\bigg(\int_0^T\Big[CQ(s)\left|v(s,f_n)\right|_{\mathbb{W}} +CQ(s)\left|z(s,f)(g)\right|_{\mathbb{W}}^2  \nonumber\\ &~~~~~~~~~~~+2\left\|\mathbb{D}\widehat{F}\big(Q(s)v(s,f_n)\big)\right\|_{L(\mathbb{V})}
 +\left|z(s,f)(g)\right|_{\mathbb{V}}^2\Big]\,ds\bigg) \nonumber\\
\leq & \bigg(C(|f|_{\mathbb{H}(\mathcal{O})}^4,T,\omega)|f_n-f|_{\mathbb{W}}^2  +\int_0^T\left\|\mathbb{D}\widehat{F}\big(Q(s)v(s,f_n)\big) -\mathbb{D}\widehat{F}\big(Q(s)v(s,f)\big)\right\|_{L(\mathbb{V})}^2\,ds\bigg) \nonumber\\
&\  \times C(|f_n|_{\mathbb{W}},|f|_{\mathbb{H}^4(\mathcal{O})},|f|_{\mathbb{W}},T,\omega).
\end{align}
\noindent Therefore, by the dominated convergence theorem, we obtain
\begin{align}
\lim_{n\rightarrow \infty}\sup_{0\leq t\leq T}\left\|z(t,f_n)-z(t,f)\right\|_{L(\mathbb{W},\mathbb{V})} =0 .
\end{align}
Since $T>0$ is arbitrary, we conclude that for any $\omega\in\Omega$ and $t\geq 0$, the map $\mathbb{W}\ni f\mapsto z(t,f,\omega)\in L(\mathbb{W},\mathbb{V})$
is continuous on the subset $\mathbb{W}\cap \mathbb{H}^4(\mathcal{O})$.

\end{proof}



\section{Proof of Theorem \ref{32.III}}

The asymptotic compactness of the random dynamical system in Theorem \ref{32.I} is established in Lemma \ref{add 0322 asymptotically compact}. According to Lemma \ref{randomattractor}, the existence of random attractors is followed immediately from Lemma \ref{add 0324 lemma.1} and Lemma \ref{add 0322 asymptotically compact}. The upper semi-continuity of random attractors is proved in Lemma \ref{add 0409.9}.

In the following Lemma \ref{add 0328.1}-\ref{add 0322 asymptotically compact}, we focus on the existence of random attractors, so we fix $\epsilon\in\mathbb{R}\backslash\{0\}$.
We first establish a supplementary property of the cocycle $u(t,f,\omega)$.
Recall (\ref{add 0328.3})-(\ref{add 0406.2}).
Consider the following equation with random coefficients,
\begin{align}\label{add 0328.1}
\left\{
\begin{aligned}
& dv(t,s,f)=\widehat{\mathcal{A}}_Q (v(t,s,f),Q(t))\,dt,\quad t> s, \\
& v(s,s,f)=f\quad \text{in}\ \mathbb{W}.
\end{aligned}
\right.
\end{align}
By the same arguments as for (\ref{4.b}), it is easy to prove the existence and uniqueness of (\ref{add 0328.1}) for every $\omega\in\Omega$.
The solution of (\ref{add 0328.1}) is denoted by $v(t,s,f,\omega)$, which is slightly different from the solution $v(t,f,\omega)$ to (\ref{4.b}), obviously $v(t,0,f,\omega)=v(t,f,\omega)$.

\begin{lem}\label{add 0328 lemma.1}
Assume {\bf(F1)}.  Let $u$ be the random dynamical system in Theorem \ref{32.I}, i.e. $u(t,f,\omega)=Q(t,\omega)v(t,f,\omega)$ is the version constructed in Section 4, then
  \begin{align}\label{add 0328.2}
    u(t,f,\theta(s,\omega))=Q(t+s,\omega)v(t+s,s,Q(s,\omega)^{-1}f,\omega), \quad\forall\, t\geq 0,\, s\in\mathbb{R}.
  \end{align}
\end{lem}

\begin{proof}
  From (\ref{add 0328.3}), we see that $Q$ also has the cocycle property
  \begin{align}\label{add 0328.4}
    Q(t+s,\omega)=Q(t,\theta(s,\omega))Q(s,\omega) .
  \end{align}
So
\begin{align}
u(t,f,\theta(s,\omega))=Q(t,\theta(s,\omega))v(t,f,\theta(s,\omega))
=Q(t+s,\omega)Q(s,\omega)^{-1}v(t,f,\theta(s,\omega)).
\end{align}
Thus to prove (\ref{add 0328.2}), it is equivalent to prove the following identity
\begin{align}\label{add 0328.5}
  Q(s,\omega)^{-1}v(t,f,\theta(s,\omega))=v(t+s,s,Q(s,\omega)^{-1}f,\omega), \quad\forall\, t\geq 0, s\in\mathbb{R}.
\end{align}

\noindent Now we define the processes
\begin{align}
  y(t,\omega):&=Q(s,\omega)^{-1}v(t,f,\theta(s,\omega)),\quad t\geq 0,\\
  \overline{y}(t,\omega):&=v(t+s,s,Q(s,\omega)^{-1}f,\omega),\quad t\geq 0.
\end{align}
Shifting the time-variable $t$ by $s$ in (\ref{add 0328.1}), we see that $\overline{y}(t,\omega)$ satisfies
\begin{align}
\begin{aligned}
  \overline{y}(t,\omega)=& Q(s,\omega)^{-1}f + \int_0^t \widehat{\mathcal{A}}_Q \Big(\overline{y}(r,\omega),Q(r+s,\omega)\Big)dr \\
  =& Q(s,\omega)^{-1}f + \int_0^t \Big[ -\nu\widehat{A}\overline{y}(r,\omega) -Q(r+s,\omega)\widehat{B}\big(\overline{y}(r,\omega)\big) \\ &~~~~~~~~~~~~~~~~~~~~~~~~+\frac{1}{Q(r+s,\omega)}\widehat{F}\big(Q(r+s,\omega)\overline{y}(r,\omega)\big)\Big] dr,\quad t\geq 0.
\end{aligned}
\end{align}
On the other hand, since $v(t,f,\omega)$ satisfies (\ref{4.b}), we have
\begin{align}\label{add 0415.1}
\begin{aligned}
  v(t,f,\theta(s,\omega))=& f+\int_0^t \widehat{\mathcal{A}}_Q \Big(v(r,f,\theta(s,\omega)),Q(r,\theta(s,\omega))\Big)dr \\
  =& f+\int_0^t \Big[-\nu\widehat{A}v(r,f,\theta(s,\omega)) -Q(r,\theta(s,\omega))\widehat{B}\big(v(r,f,\theta(s,\omega))\big) \\
  &~~~~~~~~~~~+\frac{1}{Q(r,\theta(s,\omega))}\widehat{F}\big(Q(r,\theta(s,\omega))v(r,f,\theta(s,\omega))\big)\Big] dr,\quad t\geq 0.
\end{aligned}
\end{align}
Multiplying (\ref{add 0415.1}) by $Q(s,\omega)^{-1}$, then using (\ref{add 0328.4}) and the bilinear property of $\widehat{B}$, we obtain
\begin{align}
\begin{aligned}
  y(t,\omega)=&Q(s,\omega)^{-1}f +\int_0^t \Big[-\nu\widehat{A}y(r,\omega) -Q(r+s,\omega)\widehat{B}\big(y(r,\omega)\big) \\
  &~~~~~~~~~~~~~~~~~~~~~~~~+\frac{1}{Q(r+s,\omega)}\widehat{F}\big(Q(r+s,\omega)y(s,\omega)\big)\Big]dr,\quad t\geq 0.
\end{aligned}
\end{align}
Hence $y(t,\omega)$ and $\overline{y}(t,\omega)$ satisfy the same equation.
Using a similar calculation as for (\ref{14.a}), we deduce that $y(t,\omega)=\overline{y}(t,\omega)$ for every $t\geq 0$, $\omega\in\Omega$. This proves (\ref{add 0328.5}), so (\ref{add 0328.2}) holds.

\end{proof}

%
%
%



\begin{lem}\label{add 0324 lemma.1}
Assume {\bf(F3)}. Then the random dynamical system in Theorem \ref{32.I} admits a closed bounded tempered random absorbing set in $\mathbb{W}$.
\end{lem}

\begin{proof}

Take any tempered random set $D(\omega)$. Let $f\in D(\theta(-t,\omega))$ for $t>0$. By Lemma \ref{add 0328 lemma.1}, we have
\begin{align}\label{add 0320.1}
  u(t,f,\theta(-t,\omega))=v(0,-t,Q(-t,\omega)^{-1}f,\omega).
\end{align}

\noindent From (\ref{add 0328.1}), it follows that
\begin{align}
\begin{aligned}
   d|v(r)|_{\mathbb{V}}^2=& 2 \big\langle \widehat{\mathcal{A}}_Q (v(r),Q(r,\omega)), v(r)\big\rangle\,dr \\ =&-2\nu\|v(r)\|^2 dr + 2Q(r,\omega)^{-1}\big\langle \widehat{F}(Q(r,\omega)v(r)), v(r)\big\rangle\,dr,
\end{aligned}
\end{align}

\noindent here $v(r)=v(r,-t,Q(-t,\omega)^{-1}f,\omega)$ for notational simplicity.
Using (\ref{Poincare}), (\ref{add 0329.2}), {\bf(F3)} and Young's inequality, we get
\begin{align}\label{add 0322.2}
   d|v(r)|_{\mathbb{V}}^2+\big(\frac{2\nu}{\mathcal{P}^2+\alpha}-2\frac{\mathcal{P}^2}{\mathcal{P}^2+\alpha} C_F -\varepsilon_1\big)|v(r)|_{\mathbb{V}}^2 dr\leq \frac{1}{\varepsilon_1}|F(0)|_{\mathbb{V}}^2 Q(r,\omega)^{-2}dr,
\end{align}
for any $\varepsilon_1 > 0$. Set $\lambda=\varepsilon_1 =\frac{\nu}{\mathcal{P}^2+\alpha}-\frac{\mathcal{P}^2}{\mathcal{P}^2+\alpha} C_F$. According to {\bf(F3)}, $\lambda=\varepsilon_1>0$.
Now solving (\ref{add 0322.2}) with the initial condition yields
\begin{align}\label{add 0319.2}
  |v(s,-t,Q(-t,\omega)^{-1}f,\omega)|_{\mathbb{V}}^2\leq \mathrm{e}^{-\lambda (s+t)}\frac{1}{Q(-t,\omega)^2}|f|_{\mathbb{V}}^2  + \frac{1}{\varepsilon_1}|F(0)|_{\mathbb{V}}^2\int_{-t}^s \mathrm{e}^{-\lambda (s-r)}\frac{1}{Q(r,\omega)^2}dr.
\end{align}

\noindent On the other hand,
since $\frac{d}{dr}v(r,-t,Q(-t,\omega)^{-1}f,\omega)$ and $\frac{d}{dr}v(r,f,\omega)$ satisfy the same equation,
it follows from (\ref{energy equation}) and (\ref{add 0323.2}) that 
$\frac{d}{dr}|v(r,-t,Q(-t,\omega)^{-1}f,\omega)|_{\mathbb{W}}^2$ also satisfies
\begin{align}
\begin{aligned}
  d|v(r)|_{\mathbb{W}}^2=&-\frac{2\nu}{\alpha}|v(r)|_{\mathbb{W}}^2 dr +\frac{2\nu}{\alpha}\Big({\rm curl}\big(v(r)\big),{\rm curl}\big(v(r)-\alpha\Delta v(r)\big)\Big)dr \\
  &+ 2\Big(Q(r,\omega)^{-1}\widehat{F}\big(Q(r,\omega)v(r)\big),v(r)\Big)_{\mathbb{W}}dr.
  \end{aligned}
\end{align}

\noindent Using (\ref{curl and V}), (\ref{inverse operator transform inequality}), {\bf(F3)} and Young's inequality, we have
\begin{align}\label{add 0322.3}
  d|v(r)|_{\mathbb{W}}^2+\big(\frac{2\nu}{\alpha}-\varepsilon_2-\varepsilon_3\big)|v(r)|_{\mathbb{W}}^2 dr\leq \frac{C(C_F)}{\varepsilon_2}|v(r)|_{\mathbb{V}}^2dr + \frac{C}{\varepsilon_3}|F(0)|_{\mathbb{V}}^2 Q(r,\omega)^{-2}dr,
\end{align}
for any $\varepsilon_2, \varepsilon_3 >0$. Set $\varepsilon_2=\varepsilon_3=\frac{\nu}{2\alpha}$ and $\gamma=\frac{\nu}{\alpha}$.
Then solving (\ref{add 0322.3}) with the initial condition yields
\begin{align}\label{add 0319.3}
\begin{aligned}
  &|v(0,-t,Q(-t,\omega)f,\omega)|_{\mathbb{W}}^2\\
  \leq &\mathrm{e}^{-\gamma t}\frac{1}{Q(-t,\omega)^2}|f|_{\mathbb{W}}^2 + \frac{C(C_F)}{\varepsilon_2}\int_{-t}^0 \mathrm{e}^{\gamma s}|v(s,-t,Q(-t,\omega)f,\omega)|_{\mathbb{V}}^2ds \\  &+ \frac{C}{\varepsilon_3}|F(0)|_{\mathbb{V}}^2\int_{-t}^0 \mathrm{e}^{\gamma s}\frac{1}{Q(s,\omega)^2}ds.
\end{aligned}
\end{align}

\noindent Substituting (\ref{add 0319.2}) into (\ref{add 0319.3}), together with $\gamma>\lambda>0$, we obtain
\begin{align}\label{add 0319.1}
\begin{aligned}
  &|v(0,-t,Q(-t,\omega)^{-1}f,\omega)|_{\mathbb{W}}^2\\
  \leq &\mathrm{e}^{-\gamma t}\frac{1}{Q(-t,\omega)^2}|f|_{\mathbb{W}}^2 + C(C_F)\int_{-t}^0 \mathrm{e}^{\gamma s}\Big[\mathrm{e}^{-\lambda (s+t)}\frac{1}{Q(-t,\omega)^{2}}|f|_{\mathbb{V}}^2 \\
   &+ C(C_F)|F(0)|_{\mathbb{V}}^2\int_{-t}^s \mathrm{e}^{-\lambda (s-r)}\frac{1}{Q(r,\omega)^2}dr\Big]ds
  + C|F(0)|_{\mathbb{V}}^2\int_{-t}^0 \mathrm{e}^{\gamma s}\frac{1}{Q(s,\omega)^2}ds \\
  \leq & C(C_F)\mathrm{e}^{-\lambda t}\frac{1}{Q(-t,\omega)^2}|f|_{\mathbb{W}}^2 + C(C_F)|F(0)|_{\mathbb{V}}^2\int_{-\infty}^0 \mathrm{e}^{\lambda s}\frac{1}{Q(s,\omega)^2}ds.
\end{aligned}
\end{align}


\noindent Note $\lambda=\frac{\nu-\mathcal{P}^2 C_F}{\mathcal{P}^2+\alpha}$. Set
\begin{align}\label{add 0409.3}
\begin{aligned}
r(\omega):=1+ C(C_F)|F(0)|_{\mathbb{V}}^2\int_{-\infty}^0 \mathrm{e}^{\frac{\nu-\mathcal{P}^2 C_F }{\mathcal{P}^2+\alpha} s}\frac{1}{Q(s,\omega)^2}ds.
\end{aligned}
\end{align}


\noindent Since $D(\omega)$ is tempered and $f\in D(\theta(-t,\omega))$, by the law of the iterated logarithm, we obtain that for a.e. $\omega\in\Omega$ and any $\beta>0$,
\begin{align}\label{add 0409.5}
  \mathrm{e}^{-\beta t}Q(-t,\omega)^2 d\big(D(\theta(-t,\omega))\big)\rightarrow 0, \quad \text{ as } t\rightarrow +\infty.
\end{align}
Therefore, we conclude from (\ref{add 0320.1}) and (\ref{add 0319.1}) that the random ball $B(r(\omega)):=\{x: |x|_{\mathbb{W}}^2\leq r(\omega)\}$ absorbs every tempered random set of $\mathbb{W}$, and obviously $B(r(\omega))$ is a closed bounded tempered random set.

\end{proof}

To prove the asymptotic compactness of the random dynamical system, we first give the following energy equation satisfied by $u(t,f,\omega)$, which can be easily derived from the energy equation for $v(t,f,\omega)$ in Lemma \ref{add 0323 lemma.1}.
\begin{align}\label{add 0323.3}
\begin{aligned}
  &|u(t,f,\omega)|_{\mathbb{W}}^2=Q(t,\omega)^2|v(t,f,\omega)|_{\mathbb{W}}^2 \\
  =& Q(t,\omega)^2\Big[|f|_{\mathbb{W}}^2\mathrm{e}^{-\frac{2\nu}{\alpha}t}+2\int_0^t \mathrm{e}^{-\frac{2\nu}{\alpha}(t-s)} K\big(v(s,f,\omega),Q(s,\omega)\big)ds\Big] \\
  =& Q(t,\omega)^2\Big[|f|_{\mathbb{W}}^2\mathrm{e}^{-\frac{2\nu}{\alpha}t}+2\int_0^t \mathrm{e}^{-\frac{2\nu}{\alpha}(t-s)} \widetilde{K}\big(u(s,f,\omega),Q(s,\omega)\big)ds\Big].
  \end{aligned}
\end{align}
where $\widetilde{K}\big(u(s,f,\omega),Q(s,\omega)\big)=K\big(Q(s,\omega)^{-1}u(s,f,\omega),Q(s,\omega)\big)$ and the function $K$ is given in (\ref{add 0323.2}).
This will play an important role in the proof of the following lemma.

\begin{lem}\label{add 0322 asymptotically compact}
Assume {\bf(F3)}. Then the random dynamical system in Theorem \ref{32.I} is asymptotically compact in $\mathbb{W}$.
\end{lem}

\begin{proof}
Take any tempered random set $D$.
We will show that for a.e. $\omega\in\Omega$, the sequence $\{u(t_n,f_n,\theta(-t_n,\omega))\}_{n=1}^{\infty}$ has a convergent subsequence in $\mathbb{W}$, whenever $t_n\rightarrow\infty$ and $f_n\in D(\theta(-t_n,\omega))$.

In Lemma \ref{add 0324 lemma.1}, we have proved that the random ball $B(r(\omega))$ absorbs $D$. So for a.e. $\omega$, the sequence $\{u(t_n,f_n,\theta(-t_n,\omega))\}_{n=1}^{\infty}$ is bounded in $\mathbb{W}$. Therefore, there exist $f_{\omega}\in B(r(\omega))$ and a subsequence denoted by $\{n_k(\omega,0)\}$ of $\mathbb{N}$, which depend on $\omega$, such that
\begin{align}\label{add 0323.1}
  u(t_{n_k(\omega,0)},f_{n_k(\omega,0)},\theta(-t_{n_k(\omega,0)},\omega))\rightharpoonup f_{\omega} \quad\text{weakly in }\mathbb{W} \text{ as }k\rightarrow\infty.
\end{align}
Hence to prove Lemma \ref{add 0322 asymptotically compact}, it suffices to show that there exists a subsequence of $\{n_k(\omega,0)\}$ denoted by $\{n_k(\omega)\}$, such that
\begin{align}
  |u(t_{n_k(\omega)},f_{n_k(\omega)},\theta(-t_{n_k(\omega)},\omega))|_{\mathbb{W}}\rightarrow |f_{\omega}|_{\mathbb{W}},\quad \text{ as }k\rightarrow\infty.
\end{align}
We prove this with the following three steps.

Step 1.
Let $j\in\mathbb{N}, j\geq 1$. If $\{n_k(\omega,j-1)\}$ and $f_{\theta(-(j-1),\omega)}$ are already defined, then we define $\{n_k(\omega,j)\}$ and $f_{\theta(-j,\omega)}$ as follows.
Note
\begin{align}\label{add 0407.1}
\begin{aligned}
 &u\Big(t_{n_k(\omega,j-1)}-j,f_{n_k(\omega,j-1)},\theta(-t_{n_k(\omega,j-1)},\omega)\Big) \\
  = & u\Big(t_{n_k(\omega,j-1)}-j,f_{n_k(\omega,j-1)},\theta\big(j-t_{n_k(\omega,j-1)},\theta(-j,\omega)\big)\Big),
\end{aligned}
\end{align}
and
\begin{align}\label{add 0411.1}
  f_{n_k(\omega,j-1)}\in D\big(\theta(-t_{n_k(\omega,j-1)},\omega)\big) =D\big(\theta(j-t_{n_k(\omega,j-1)},\theta(-j,\omega))\big) .
\end{align}

\noindent Since $B(r(\omega))$ absorbs $D$, the sequence $\{u(t_{n_k(\omega,j-1)}-j,f_{n_k(\omega,j-1)},\theta(-t_{n_k(\omega,j-1)},\omega))\}_{k=1}^{\infty}$ is bounded in $\mathbb{W}$. Hence there exist a subsequence of $\{n_k(\omega,j-1)\}$, for notational simplicity we write it as $\{n_k(\omega,j)\}$, and $f_{\theta(-j,\omega)}\in B\big(r(\theta(-j,\omega))\big)$  such that
\begin{align}\label{add 0320.4}
  u\Big(t_{n_k(\omega,j)}-j,f_{n_k(\omega,j)},\theta(-t_{n_k(\omega,j)},\omega)\Big)\rightharpoonup f_{\theta(-j,\omega)} \quad\text{weakly in }\mathbb{W} \text{ as }k\rightarrow\infty .
\end{align}

\noindent Now using the diagonal method, we construct a sequence as $n_k(\omega):=n_k(\omega,k)$.

%
%
%
%
%
%

Step 2.  By the cocycle property of the random dynamical system, we have for any $\omega\in\Omega$, $j\in\mathbb{N}$ and $t_{n_k(\omega)}\geq j$,
%
%
\begin{align}\label{add 0320.2}
\begin{aligned}
 & u\Big(j, u\big(t_{n_k(\omega)}-j,f_{n_k(\omega)},\theta(-t_{n_k(\omega)},\omega)\big), \theta(-j,\omega)\Big)\\
  = & u\Big(j, u\big(t_{n_k(\omega)}-j,f_{n_k(\omega)},\theta(-t_{n_k(\omega)},\omega)\big), \theta\big(t_{n_k(\omega)}-j,\theta(-t_{n_k(\omega)},\omega)\big)\Big)\\
  = & u\Big(t_{n_k(\omega)},f_{n_k(\omega)},\theta(-t_{n_k(\omega)},\omega)\Big).
 \end{aligned}
\end{align}
From (\ref{34.a}) and Remark \ref{add 0320 remark}, it follows that for any $\omega\in\Omega$ and $t\geq 0$, $u(t,f,\omega)$ is weakly continuous in $f\in\mathbb{W}$.
Due to (\ref{add 0320.4}), we deduce that the first line of (\ref{add 0320.2}) is weakly convergent to $u(j,f_{\theta(-j,\omega)},\theta(-j,\omega))$ as $k\rightarrow\infty$. On the other hand, (\ref{add 0323.1}) implies that the last line of (\ref{add 0320.2}) is weakly convergent to $f_{\omega}$ as $k\rightarrow\infty$. Therefore, we obtain that for a.e. $\omega\in\Omega$,
\begin{align}\label{add 0320.6}
  u\big(j,f_{\theta(-j,\omega)},\theta(-j,\omega)\big)=f_{\omega},\quad \forall\, j\in\mathbb{N}.
\end{align}



Step 3.  In view of (\ref{add 0320.2}) and the energy equation  (\ref{add 0323.3}), we get for any $\omega\in\Omega$, $j\in\mathbb{N}$ and $t_{n_k(\omega)}\geq j$,
\begin{align}\label{add 0320.5}
\begin{aligned}
  &\Big| u\Big(t_{n_k(\omega)},f_{n_k(\omega)},\theta(-t_{n_k(\omega)},\omega)\Big)\Big|_{\mathbb{W}}^2 \\
  =& Q\big(j,\theta(-j,\omega)\big)^2\bigg[\big|u\big(t_{n_k(\omega)} -j,f_{n_k(\omega)},\theta(-t_{n_k(\omega)},\omega)\big)\big|_{\mathbb{W}}^2\mathrm{e}^{-\frac{2\nu}{\alpha}j} \\
  & +2\int_0^j \mathrm{e}^{-\frac{2\nu}{\alpha}(j-s)} \widetilde{K}\bigg(u\Big(s,u\big(t_{n_k(\omega)} -j,f_{n_k(\omega)},\theta(-t_{n_k(\omega)},\omega)\big),\theta(-j,\omega)\Big),Q\big(s,\theta(-j,\omega)\big)\bigg)ds\bigg].
\end{aligned}
\end{align}

\noindent Since the random ball $B(r(\omega))$ absorbs $D$, similar to (\ref{add 0407.1}) and (\ref{add 0411.1}), we see that
\begin{align}\label{add 0323.5}
  \limsup_{k\rightarrow\infty}|u\big(t_{n_k(\omega)} -j,f_{n_k(\omega)},\theta(-t_{n_k(\omega)},\omega)\big)\big|_{\mathbb{W}}^2\leq r(\theta(-j,\omega)).
\end{align}


\noindent By (\ref{add 0320.4}), similar to (\ref{continuous convergence1})-(\ref{continuous convergence4}), we can show that for any $\omega\in\Omega$, $j\in\mathbb{N}$,
\begin{align}\label{add 0323.4}
\begin{aligned}
   u\Big(\cdot,u\big(t_{n_k(\omega)} -j,f_{n_k(\omega)},\theta(-t_{n_k(\omega)},\omega)\big),\theta(-j,\omega)\Big)
  \rightarrow  u\big(\cdot,f_{\theta(-j,\omega)},\theta(-j,\omega)\big),
\end{aligned}
\end{align}
weakly in $L^{2}([0,T];\mathbb{W})$ and strongly in $L^{2}([0,T];\mathbb{V})$ as $k\rightarrow\infty$.
Therefore, letting $k\rightarrow\infty$ in (\ref{add 0320.5}), together with (\ref{add 0323.5}) and (\ref{add 0323.4}) yields
\begin{align}\label{add 0320.7}
\begin{aligned}
  &\limsup_{k\rightarrow\infty}\Big| u\Big(t_{n_k(\omega)},f_{n_k(\omega)},\theta(-t_{n_k(\omega)},\omega)\Big)\Big|_{\mathbb{W}}^2 \\
  \leq & Q\big(j,\theta(-j,\omega)\big)^2\bigg[r\big(\theta(-j,\omega)\big)\mathrm{e}^{-\frac{2\nu}{\alpha}j} \\ &+2\int_0^j \mathrm{e}^{-\frac{2\nu}{\alpha}(j-s)} \widetilde{K}\Big( u\big(s,f_{\theta(-j,\omega)},\theta(-j,\omega)\big),Q\big(s,\theta(-j,\omega)\big)\Big)ds\bigg] .
  \end{aligned}
\end{align}

\noindent On the other hand,  it follows from (\ref{add 0320.6}) and (\ref{add 0323.3}) that for a.e. $\omega\in\Omega$,
\begin{align}\label{add 0320.8}
\begin{aligned}
\big|f_{\omega}\big|_{\mathbb{W}}^2 =& \big|u\big(j,f_{\theta(-j,\omega)},\theta(-j,\omega)\big)\big|_{\mathbb{W}}^2 \\
=& Q\big(j,\theta(-j,\omega)\big)^2\bigg[\big|f_{\theta(-j,\omega)}\big|_{\mathbb{W}}^2 \mathrm{e}^{-\frac{2\nu}{\alpha}j} \\  &+2\int_0^j \mathrm{e}^{-\frac{2\nu}{\alpha}(j-s)} \widetilde{K}\Big( u\big(s,f_{\theta(-j,\omega)},\theta(-j,\omega)\big),Q\big(s,\theta(-j,\omega)\big)\Big)ds\bigg], \quad\forall\,j\in\mathbb{N}.
\end{aligned}
\end{align}

\noindent Subtracting (\ref{add 0320.8}) from (\ref{add 0320.7}) gives
\begin{align}\label{add 0323.6}
\begin{aligned}
  &\limsup_{k\rightarrow\infty}\Big|u\Big(t_{n_k(\omega)},f_{n_k(\omega)},\theta(-t_{n_k(\omega)},\omega)\Big)\Big|_{\mathbb{W}}^2 -\big|f_{\omega}\big|_{\mathbb{W}}^2 \\
  \leq & Q\big(j,\theta(-j,\omega)\big)^2\Big[r\big(\theta(-j,\omega)\big) - |f_{\theta(-j,\omega)}\big|_{\mathbb{W}}^2 \Big]\mathrm{e}^{-\frac{2\nu}{\alpha}j} , \quad\forall\, j\in\mathbb{N} .
\end{aligned}
\end{align}
Since $f_{\theta(-j,\omega)}\in B\big(r(\theta(-j,\omega))\big)$ and $B(r(\omega))$ is tempered, by the law of the iterated logarithm, we deduce that the right hand side of (\ref{add 0323.6}) tends to $0$ as $j\rightarrow +\infty$. Therefore, we obtain
\begin{align}
  \limsup_{k\rightarrow\infty}\Big|u\Big(t_{n_k(\omega)},f_{n_k(\omega)},\theta(-t_{n_k(\omega)},\omega)\Big)\Big|_{\mathbb{W}}^2 \leq\big|f_{\omega}\big|_{\mathbb{W}}^2.
\end{align}
On the other hand, in view of (\ref{add 0323.1}), we have
\begin{align}
  \big|f_{\omega}\big|_{\mathbb{W}}^2 \leq \liminf_{k\rightarrow\infty}\Big|u\Big(t_{n_k(\omega)},f_{n_k(\omega)},\theta(-t_{n_k(\omega)},\omega)\Big)\Big|_{\mathbb{W}}^2.
\end{align}
Hence for a.e. $\omega\in\Omega$, the subsequence $\{u\big(t_{n_k(\omega)},f_{n_k(\omega)},\theta(-t_{n_k(\omega)},\omega)\big)\}_{k=1}^{\infty}$ is strongly convergent to $f_{\omega}$ in $\mathbb{W}$. This completes the proof of Lemma \ref{add 0322 asymptotically compact}.

\end{proof}

Finally, we present the upper semi-continuity of the random attractors as $\epsilon\rightarrow 0$. To this end, we first establish the convergence of the solution of (\ref{Abstract}) as $\epsilon\rightarrow 0$.

In the following two lemmas, To indicate the dependence on $\epsilon$, we write the solution $u$ in Theorem \ref{32.I}, the notation $Q$ in (\ref{add 0328.3}), the solution $v$ of (\ref{4.b}) as  $Q^{\epsilon}, u^{\epsilon}, v^{\epsilon}$ respectively.
We denote by the map $u^0: \mathbb{R}^+ \times \mathbb{W}\rightarrow\mathbb{W}$ the deterministic continuous dynamical system generated by the solutions $u^0(t,f)$ of equation (\ref{Abstract}) with $\epsilon=0$.
\begin{lem}\label{add 0409.6}
Assume {\bf(F1)}. Let $\{\epsilon_n\}_{n=1}^{\infty}$ be any positive sequence such that $\epsilon_n\rightarrow \epsilon$ for some $\epsilon\in\mathbb{R}$. Take any sequence $\{f_n\}_{n=1}^{\infty}\subset\mathbb{W}$ and $f\in\mathbb{W}$. If $f_n\rightarrow f$ strongly in $\mathbb{W}$ as $n\rightarrow\infty$, then for any $\omega\in\Omega$ and $t\geq 0$,
  \begin{align}\label{add 0412.1}
    u^{\epsilon_n}(t,f_n,\omega)\rightarrow u^{\epsilon}(t,f,\omega) \quad\text{ strongly in } \mathbb{W} \text{ as } n\rightarrow\infty .
  \end{align}
If $f_n\rightharpoonup f$ weakly in $\mathbb{W}$ as $n\rightarrow\infty$, then for any $\omega\in\Omega$ and $t\geq 0$,
\begin{align}\label{add 0412.2}
    u^{\epsilon_n}(t,f_n,\omega)\rightharpoonup u^{\epsilon}(t,f,\omega) \quad\text{ weakly in } \mathbb{W} \text{ as } n\rightarrow\infty .
  \end{align}
In (\ref{add 0412.1}) and (\ref{add 0412.2}), we set $u^0(t,f,\omega):=u^0(t,f)$ as $\epsilon=0$.
\end{lem}

\begin{proof}
  Note
  \begin{align}
    u^{\epsilon}(t,f,\omega)=Q^{\epsilon}(t,\omega)v^{\epsilon}(t,f,\omega), \quad t\geq 0,
  \end{align}
  and $Q^{\epsilon_n}(t,\omega)\rightarrow Q^{\epsilon}(t,\omega)$ as $\epsilon_n\rightarrow\epsilon$. Hence it suffices to prove
  \begin{align}
    f_n\rightarrow f \text{ strongly in } \mathbb{W} &\Rightarrow v^{\epsilon_n}(t,f_n,\omega)\rightarrow v^{\epsilon}(t,f,\omega)  \text{ strongly in } \mathbb{W}, \\
    f_n\rightharpoonup f \text{ weakly in } \mathbb{W} &\Rightarrow v^{\epsilon_n}(t,f_n,\omega)\rightharpoonup v^{\epsilon}(t,f,\omega)  \text{ weakly in } \mathbb{W},
  \end{align}
where we set $v^0(t,f,\omega):=u^0(t,f)$ as $\epsilon=0$.
Without loss of generality, we may assume $|\epsilon_n|\leq |\epsilon|+1$ for every $n\in\mathbb{N}$. Since $\sup_n |f_n|_{\mathbb{W}}^2\leq C<\infty$,
by the same arguments as for (\ref{W norm estimate of v}), we have
\begin{align}
\sup_{n\in\mathbb{N}}\sup_{t\in[0,T]}\big|v^{\epsilon_n}(t,f_n,\omega)\big|_{\mathbb{W}}^2\leq C(T)\bigg(C+|F(0)|_{\mathbb{V}}^2\int_0^T \mathrm{e}^{2(|\epsilon|+1)|\omega(s)|}\,ds\bigg) .
\end{align}

\noindent Using the similar method as for Proposition \ref{16.6.I} and Remark \ref{add 0320 remark}, Lemma \ref{add 0409.6} can be proved.

\end{proof}

Let $\mathcal{A}^{\epsilon}$ be the random attractor for the random dynamical system $u^{\epsilon}$ in Theorem \ref{32.I}.
Under the assumption {\bf(F3)}, the existence of the global attractor $\mathcal{A}^0$ in $\mathbb{W}$ for the deterministic dynamical system $u^0$ can similarly be achieved by Lemma \ref{add 0324 lemma.1} and Lemma \ref{add 0322 asymptotically compact}.
\begin{lem}\label{add 0409.9}
Assume {\bf(F3)}. Then the family of the random attractors $\{\mathcal{A}^{\epsilon}\}_{|\epsilon|>0}$ is upper semi-continuous as $\epsilon\rightarrow 0$, i.e. for a.e. $\omega\in\Omega$,
\begin{align}
  d(\mathcal{A}^{\epsilon}(\omega),\mathcal{A}^0)\rightarrow 0 \quad\text{ as } \epsilon\rightarrow 0,
\end{align}
where $d$ is the Hausdorff semi-metric defined in (\ref{add 0410.1}).
\end{lem}

\begin{proof}
It suffices to show that the conditions (i)-(iii) in Lemma \ref{add 0409.8} are satisfied.
By Lemma \ref{add 0409.6}, the condition (i) is satisfied.

In Lemma \ref{add 0324 lemma.1}, we have proved that for each $\epsilon\in\mathbb{R}\backslash\{0\}$, the random dynamical system $u^{\epsilon}$ possesses an absorbing random ball $B(r^{\epsilon}(\omega))$, where
\begin{align}\label{add 0409.4}
\begin{aligned}
   r^{\epsilon}(\omega):=& 1+C(C_F)|F(0)|_{\mathbb{V}}^2\int_{-\infty}^{0} \mathrm{e}^{\frac{\nu-\mathcal{P}^2 C_F }{\mathcal{P}^2+\alpha} s}\frac{1}{{Q^{\epsilon}(s,\omega)}^2} ds \\
  \leq & 1+C(C_F)|F(0)|_{\mathbb{V}}^2\int_{-\infty}^{0} \mathrm{e}^{\frac{\nu-\mathcal{P}^2 C_F }{\mathcal{P}^2+\alpha} s + 2|\epsilon||\omega(s)|} ds .
\end{aligned}
\end{align}
Hence the dominated convergence theorem yields
\begin{align}
  \limsup_{\epsilon\rightarrow 0}r^{\epsilon}(\omega)\leq 1+C(C_F)|F(0)|_{\mathbb{V}}^2,
\end{align}
which deduce the condition (ii) in Lemma \ref{add 0409.8} immediately.

For the condition (iii), we first claim that for any tempered random set $D$, there exists a constant $\widetilde{T}>0$, such that
\begin{align}\label{add 0409.7}
  \sup_{0<\epsilon<1}|u^{\epsilon}(t,f,\omega)|_{\mathbb{W}}\leq \rho(\omega), \quad\forall\,t>\widetilde{T} ,f\in D(\theta(-t,\omega)) ,
\end{align}
where
\begin{align}
  \rho(\omega):=1+C(C_F)|F(0)|_{\mathbb{V}}^2\int_{-\infty}^{0} \mathrm{e}^{\frac{\nu-\mathcal{P}^2 C_F }{\mathcal{P}^2+\alpha} s + 2|\omega(s)|} ds > r^{\epsilon}(\omega), \quad\forall\,\epsilon\in (-1,1).
\end{align}
In other words, the tempered random ball $B(\rho(\omega)):=\{x: |x|_{\mathbb{W}}^2\leq\rho(\omega)\}$ absorbs every tempered random set of $\mathbb{W}$ uniformly with respect to $\epsilon\in(-1,1)$.
Claim (\ref{add 0409.7}) is followed easily from (\ref{add 0409.5}), where the convergence is uniform with respect to $\epsilon\in(-1,1)$.

Now we come to show the condition (iii) in Lemma \ref{add 0409.8}.
Let $\{u_n\}_{n=1}^{\infty}\subseteq\cup_{0<|\epsilon|<1}\mathcal{A}_{\epsilon}(\omega)$.
Then there exists a sequence $\{\epsilon_n\}_{n=1}^{\infty}\subset(-1,1)\backslash\{0\}$ such that $u_n\in\mathcal{A}^{\epsilon_n}(\omega)$ for each $n\in\mathbb{N}$.
Hence we can take a subsequence $\{\epsilon_{n_k}\}_{k=1}^{\infty}$ such that $\epsilon_{n_k}\rightarrow\epsilon$ for some $\epsilon\in[-1,1]$ as $k\rightarrow\infty$, for notational simplicity we still write this subsequence as $\{\epsilon_n\}_{n=1}^{\infty}$ in the following.
The invariance of the random attractors implies that
\begin{align}
  u^{\epsilon}\big(t,\mathcal{A}^{\epsilon}(\theta(-t,\omega)),\theta(-t,\omega)\big)=\mathcal{A}^{\epsilon}(\omega), \quad\forall\,t\geq 0.
\end{align}
Therefore, there exist sequences $\{t_n\}_{n=1}^{\infty}$ and $\{f_n\}_{n=1}^{\infty}$ such that 
\begin{align}
t_n\rightarrow +\infty, \quad f_n\in\mathcal{A}^{\epsilon_n}(\theta(-t_n,\omega)),\quad  u_n=u^{\epsilon_n}(t_n,f_n,\theta(-t_n,\omega)) .
\end{align}
Again, by the invariance of random attractors, we have
\begin{align}
  \mathcal{A}^{\epsilon}(\omega)\subseteq B(r^{\epsilon}(\omega))\subseteq B(\rho(\omega)),\quad\forall\, \epsilon>0.
\end{align}
Hence
\begin{align}
  f_n\in B\big(\rho(\theta(-t_n,\omega))\big), \quad \forall\, n\in\mathbb{N} .
\end{align}
Due to claim (\ref{add 0409.7}), the sequence $\{u^{\epsilon_n}(t_n,f_n,\theta(-t_n,\omega))\}_{n=1}^{\infty}$ is bounded in $\mathbb{W}$.
Now using claim (\ref{add 0409.7}) and Lemma \ref{add 0409.6}, by the same arguments as for Lemma \ref{add 0322 asymptotically compact}, we obtain that the sequence $\{u^{\epsilon_n}(t_n,f_n,\theta(-t_n,\omega))\}_{n=1}^{\infty}$ has a subsequence which is convergent in $\mathbb{W}$. Hence the condition (iii) in Lemma \ref{add 0409.8} is satisfied, which completes the proof.

%

\end{proof}


\section*{Acknowledgements}
The author sincerely thanks Professor Tusheng Zhang and Jianliang Zhai for their instructions and many invaluable suggestions. This work is partly supported by National Natural Science Foundation of China (No.11671372, No.11431014, No.11721101).

\end{document}